\documentclass[a4paper,11pt]{article}
\usepackage{amsfonts}
\usepackage{latexsym}
\usepackage{amsthm}
\usepackage{amssymb}
\usepackage{amsmath}
\usepackage{eufrak}
\usepackage{mathrsfs}
\usepackage{geometry}
\usepackage{graphics}
\usepackage{comment}
\usepackage[all]{xy}
\usepackage{rotating}
\usepackage{algorithmic}
\usepackage{algorithm}
\usepackage{color, colortbl}
\usepackage{xcolor}
\usepackage{bm}
\usepackage{mathtools}
\usepackage{varwidth}
\usepackage{lipsum}
 \usepackage{enumitem}

\makeatletter
    \newcommand\iorarrow[2][n]{%
        \mathpalette{\overarrow@{\arrowfill@\relbar\relbar%
            {\rightarrow\mkern -2mu\smash{\mathrlap{_{#1}}}}%
        }}%
       {\,#2\,}%
        \mkern -2mu\hphantom{_{#1}}%
    }
\makeatother

\newtheorem{theo}{Theorem}
\newtheorem{definition}{Definition}
\newtheorem{Lemma}{Lemma}
\newtheorem{Cor}{Corollary}
\newtheorem{Prop}{Proposition}[section]

\newtheorem{Rm}{Remark}

\newtheorem{Ex}{Example}

\newtheorem{notation}{Notation}

\newtheorem*{pps1}{Proposition \ref{nt-minimal}}
\newtheorem*{tmaF}{Theorem \ref{final}}
\newtheorem*{tma2p}{Theorem \ref{iff2pt}}

\newcommand{\lra}{ \longrightarrow }

\author{Natalia A. Viana Bedoya and  Daciberg Lima Gon\c calves}

\title{Decomposability of minimal defect branched coverings over the projective plane}
\date{}
\begin{document}
\maketitle

\begin{abstract}In this paper we characterize  primitive branched coverings with minimal defect over the projective plane with respect to the properties   decomposable and indecomposable. This minimality is achieved when the covering surface is also the projective plane, which corresponds to the last case to be solved. As a consequence,  we obtain a larger class of realizations  than in   \cite{EKS} and   we also  establish a kind of generalization of some results in \cite{ACDHL} and \cite{Lo} about primitive permutation groups. \\


\end{abstract}

\section*{Introduction}\label{intro}
Given a  branched covering $\phi: M \rightarrow N$ of degree $d$ between closed connected surfaces, there exists
a finite collection of partitions $\mathscr{D}$ of $d$, \emph{the branch
 datum}, in correspondence with the {\it branch point set} $B_{\phi} \subset N$.  The {\it total defect} $\nu(\mathscr{D})$ associated to $\phi$  is an integer characterized  by the properties:   \\
 \centerline{$ \nu(\mathscr{D})=d \chi(N)-\chi(M) \equiv 0 \pmod{2}.$}
 Moreover, if $N \neq S^{2}$, a finite collection of partitions $\mathscr{D}$ of $d$ satisfying the properties above is realizable
  as a branch datum of  a branched covering of degree $d$  of the form $ M \rightarrow N$. For more 
details see \cite{EKS}, \cite{Ez}, \cite{Hu}.\\
 Collections of partitions satisfying these conditions will be called {\it admissible data.} 
 
 A branched covering   is called
\emph{decomposable} if it can be written as a composition of two non-trivial
coverings (both with degree bigger than 1), otherwise it is 
\emph{indecomposable}. \\
In this work, we study the decomposability property  in terms of   admissible data for
 primitive (surjective on $\pi_{1}$) branched coverings,  since it is  known that the non-primitive ones  are  decomposable (see \cite{BGZ}). We will call {\it decomposable data on $N$} the  admissible data realizable by 
 decomposable primitive branched coverings over a surface $N$.
  
 \vspace{5pt}
This property is already well understood for $\chi(N) \leq 1$ and $\chi(M)\leq0$. The main results are:
 
 \vspace{5pt}
\noindent {\it Proposition 2.6 \cite{BN&GDL1}:
Admissible data $\mathscr{D}$ are decomposable data on $N, $ with $\chi(N)\leq 0$, 
 if and only if there exists a factorization of $\mathscr{D}$
  such that its first factor is a  non-trivial  admissible datum.}

 \vspace{7pt}
\noindent {\it Theorem 3.3 \cite{BN&GDL1}:\label{N<0}
All non-trivial admissible data are realizable on any $N$, with $\chi(N)\leq 0$, by an
indecomposable  branched covering.}

 \vspace{7pt}
\noindent {\it Theorem 3.7 \cite{BN&GDL2}:\label{even}
 Let $d$ be even and $\mathscr{D}=\{D_{1},\dots,D_{s}\}$ an admissible datum. Then, $\mathscr{D}$ is realizable by an indecomposable
   branched covering over $\mathbb{R}P^{2}$ if, and only if, either:\\
(1) $d=2$, or\\
 (2) There is  $i \in \{1,\dots,s\}$ such that
   $D_{i} \neq [2,\dots,2]$, or\\
(3) $d>4$ and $s>2$.}

 \vspace{7pt}
\noindent {\it Theorem 4.6 \cite {BN&GDL3}: \label{bigger}
Let  $\mathscr{D}$ be an admissible datum such that $\nu(\mathscr{D})>d$. Then it can be realized as branch datum of an indecomposable branched covering of degree $d$ over the projective plane.}

 \vspace{5pt}
\noindent By the results above, we see that there exist  three classes of admissible data: the class of data realizable only by decomposable branched coverings (for example the complementar cases in Theorem 3.7 \cite{BN&GDL2}), the class of data
 realizable only by indecomposable branched coverings (for example those that do not satisfy  Proposition 2.6 \cite{BN&GDL1}) and the class of data realizable by both, decomposable and indecomposable branched coverings (for example those data which are decomposable and satisfy the hypothesis of either Theorem 3.3 \cite{BN&GDL1} or Theorem 4.6 \cite {BN&GDL3}).

\vspace{5pt}
In this work we study the remaining case. i.e. $N=\mathbb{R}P^2$,  $d$ odd  and  $\nu(\mathscr{D})=d-1$. This case corresponds  precisely to the branched coverings where the covering space 
is also the projective plane.

\vspace{5pt}
\noindent The paper is divided into  four sections. In Section \S \ref{PTN} we quote the main
definitions and some results related to branched coverings.
\noindent In Section \S\ref{decomp} we characterize decomposable data for this case and
 the main results  is:
 
 \vspace{-3pt}
\begin{pps1}
An admissible datum $\mathscr{D}$ with minimal defect  on $\mathbb{R}P^2$   is decomposable if and only if
  there exists an algebraic factorization of
 $\mathscr{D}$ such that its first factor 
is  a non-trivial admissible datum on $\mathbb{R}P^2$ with minimal defect.
\end{pps1}
\noindent Section \S \ref{leq2} is the main and most technical section of the work. There we study the case of data with two partitions  through the construction of (im)primitive permutation groups, via Proposition \ref{yo} (see \S \ref{decomp}). 
These constructions    are interesting in their own right. As we shall see, they allow us to  obtain a larger class of realizations  than  in   \cite{EKS} (see  Lemma \ref{3cycles0}  and Proposi\c{c}\~ao \ref{N}) and  we also  establish a sort of generalization of some results in \cite{ACDHL} and \cite{Lo} (see \S \ref{UD}),  which are works interested in recognising primitive and imprimitive permutations, as defined in \cite{ACDHL} (see Remark 1 in section 1).\\

\noindent  To state the main result of these section we need the following notation:
 
 \vspace{3pt}
\noindent {\it Given a partition $D=[k_1,\dots,k_r]$ of an integer $d$, we denote by $\gcd(D)$ the greater common divisor of the $k_i$'s, i.e.
$\gcd(D):=\gcd\{k_1,\dots,k_r\}.$}

\vspace{-2pt}
\begin{tma2p}Let $\mathscr{D}=\{D_{1},D_{2}\}$  be a branch datum such that $\nu(\mathscr{D})=d-1$. Then $\mathscr{D}$ is realizable by an indecomposable branched covering if and only if  $\gcd(D_{i})=1$ for $i=1,2$ and  $\mathscr{D}\neq \{[2,\dots,2,1],[2,\dots,2,1]\} $.
\end{tma2p}
\noindent In Section \S \ref{arbpt}  we study the case  with more than two partitions and the main result is:

\vspace{-2pt}
\begin{tmaF}
 For $k\geq3$, let  $\mathscr{D}=\{D_1, \dots, D_k\}$  be a collection of partitions of $d \in \mathbb{Z}$   odd  non-prime,  such that $\nu(\mathscr{D})=d-1$. Then $\mathscr{D}$ is realizable by an indecomposable branched covering if and only if $\gcd(D_i)=1$ for $i=1, \dots, k$.
\end{tmaF}
\noindent In the Appendix we show   generalisations of some results in \S \ref{leq2}, which are going to be used in \S 4.

\section{Preliminaries}\label{PTN}
  \subsection{On permutation groups}\label{PG}
The symmetric group of a set $\Omega$ with $d$ elements will be denote by $\Sigma_{d}$ and 
 by $1_{d}$ its identity element. A subgroup $G < \Sigma_{d}$ will be called {\it permutation group of degree $d$}. 
 
  The \emph{cyclic decomposition} of a permutation $\alpha$ is the factorization of $\alpha$ as a product of disjoint cycles. The set of  lengths  of these cycles,  including the trivial ones, defines a partition of $d$, 
 $D_{\alpha}=[d_{\alpha_{1}},\dots,d_{\alpha_{t}}]$, called
 \emph{the cyclic structure of} $\alpha$. \\
 We say that $\alpha$ is an \emph{even permutation} if $\nu(\alpha):=\sum_{i=1}^{t}(d_{\alpha_{i}}-1) \equiv 0 \pmod{2}$. Given  a partition $D$ of $d$, we write  $\alpha \in
D$ if 
$D=D_{\alpha}$ and  we define $\nu(D):=\nu(\alpha)$.\\
A  $r$-\emph{cycle}, for $1<r \leq d$, is permutation  such that its cyclic decomposition has a single non-trivial cycle of length $r$. \\
Permutations $\alpha,\beta \in \Sigma_{d}$ are \emph{conjugate}
if there is $\lambda \in \Sigma_{d}$  such that
$\alpha^{\lambda}:=\lambda \alpha \lambda^{-1}=\beta$, which is equivalent to $\alpha$ and $\beta$ having
 the same cyclic structure.

Let $G$ be a permutation group of degree $d$, $\alpha \in G$   and $x \in \Omega$. We denote by
$x^{\alpha}$  the image of $x$ by $\alpha$, by 
 $G_{x}:=\{\alpha \in G: x^{\alpha}=x\}$   the {\it isotropy subgroup of $x$} and by
  $x^{G}:=\{x^{\alpha}:\alpha\in G\}$ the {\it orbit of $x$ by $G$}. \\
  For $H \subset G$ and $\Lambda \subset \Omega$ , we define the subsets $Supp(H):=\{x
 \in \Omega: x^{\mu} \neq x \textrm{ for some $\mu \in H$}\}$ the {\it support  of $H$},
 $Fix(H):=\{x \in \Omega:x^{\mu}=x \textrm{ for all $\mu \in H$}\}$ and   $\Lambda^{\alpha}:=\{y^{\alpha}:y \in \Lambda\}$.\\
A nonempty subset $\Lambda \subset \Omega$ is a \emph{block} of
 $G$ if
 for each $\alpha \in G$ either $\Lambda^{\alpha} = \Lambda$ or
 $\Lambda^{\alpha} \cap \Lambda =\emptyset$. A block $\Lambda$ is
 \emph{trivial} if either $\Lambda =\Omega$ or $\Lambda=\{ x \}$ for some $x \in
 \Omega$.
 
 $G$ is 
 \emph{transitive} if for all
 $x,y \in \Omega$ there is $\alpha \in G$ such that
 $x^{\alpha}= y$.
  Given a block $\Lambda$ of a transitive $G$, the set
 $\Gamma:=\{\Lambda^{\alpha}:\alpha \in G\}$ defines a partition of $\Omega$
 in blocks, \emph{a system of blocks
 containing} $\Lambda$, and the cardinality of $\Lambda$ divides the cardinality of $\Omega$. 
 
 $G$ transitive  is
 \emph{primitive}
 if $G$ has no nontrivial blocks on $\Omega$. Otherwise $G$ is called
 \emph{imprimitive}. A classical characterization of this property is given by the following proposition:
 
\begin{Prop}[Cor. 1.5A, \cite{DM}]\label{dixon}
Let $G$ be a transitive permutation group on a set $\Omega$ with at least two
points. Then $G$ is primitive if and only if each isotropy subgroup $G_{x}$,
for $x \in \Omega$,
is a maximal subgroup of $G$. \qed
\end{Prop}
  The following Lemma is a practical criterion for recognising the primitivity of a permutation group in terms of the  existence of elements in the group with 
certain  cyclic structure.  Moreover, it partially generalises   Theorem 5.4 of \cite{ACDHL}, see Remark \ref{GEN1}.
\begin{comment}
\begin{Ex}\label{e1}
A transitive permutation group $G<\Sigma_{d}$ containing a  $(d-1)$-cycle is
  primitive.  Without loss of generality let us suppose that
 $g=(1\dots d-1)(d) \in G$. Then
 any  proper subset $\Lambda$ of $\{1,\dots,d\}$ containing $d$ and at least one more element satisfies $\Lambda^{g}\neq \Lambda$ and
$\Lambda^{g}\cap \Lambda \neq \emptyset$. Thus the blocks
of $G$ are
trivial and $G$ is primitive.
\end{Ex}
\begin{Ex}\label{referee}
If $gcd(l,d)=1$ and $l$ is greater than any non-trivial divisor of $d$ then any transitive permutation group $G < \Sigma_{d}$ containing a $l$-cycle is primitive
(this holds, for example, if $d=2l\pm1$). In fact, we can assume that $G$ contains the cycle $(1 \dots l)$. If there is a block of $G$ containing two elements $i$ and $j$
with $i \leq l$ and $j>l$ then it also contains $1,\dots,l$, thus the cardinality of the block is $\geq l+1$, hence it  equals $d$ and the block is trivial.
 Otherwise the cardinality of each block divides both $l$ and $d-l$, hence it equals 1, thus all blocks of $G$ are trivial. Hence $G$ is a primitive permutation group.
\end{Ex}
\end{comment}

\begin{Lemma}\label{gcd}
A transitive permutation group $G$ of degree $d$ odd, containing a  permutation with cyclic structure $[d-2c,c,c]$ for  $gcd(d,c)=1$, is primitive. 
\end{Lemma}
\begin{proof}
Let $G$ be a transitive permutation group which contains a permutation $\lambda$ such that  $\lambda \in [d-2c,c,c]$.
By contradiction, suppose $G$ imprimitive and let $\Lambda$ be  a non-trivial block of $G$.
  If $\Lambda$ is defined only by elements in the $(d-2c)$-cycle necessarily $gcd(d-2c, d)=k\neq 1$ and  there are  integers $0<s<t$ such that $d=kt$ and $d-2c=ks$. Hence $2c=k(t-s)$ and since $k$ is odd (because $d$ is odd) then $k$ divides $c$, a contradiction because $gcd(d,c)=1.$\\
 Since   there exists a block  $\Lambda_1$ which contains an element of  the $(d-2c)$-cycle, from above this block must contains an element  of 
the  $(c)$-cycle. 
 By applying $\lambda$  to the block $\Lambda_1$  several times, the block will absorb both cycles, because $gcd(d-2c,c)=1$. Then $|\Lambda|\geq d-c.$ But  $|\Lambda|>d-c$ implies that elements in the other $(c)$-cycle are in $\Lambda$ and again, by applying $\lambda$ several times, $\Lambda$ will absorb the cycle and the block will be  trivial. If $|\Lambda|= d-c$ then $d-c$ divides $d$ and therefore $d-c$ divides $c$. Hence $d-c \leq c$ and $d-2c \leq 0$, a contradiction.
\end{proof}

\begin{Rm}\label{GEN1}In \cite{ACDHL}, a  permutation is said to be imprimitive if there exists an imprimitive group containing it. Otherwise, the permutation is primitive.  Theorem 5.4 of \cite{ACDHL} states that a permutation with cyclic structure $[1, 1, n-2]$ is primitive if and only if $n$ is odd.  Lemma \ref{gcd}
generalises the ``if" part of the result above.  Also the converse of the Lemma \ref{gcd} does not hold. 
\end{Rm}

\subsection{Branched coverings on the projective plane}\label{pbc}
A surjective continuous open map $\phi:M \lra N$ between closed surfaces such
that:
 \vspace*{-7pt}
\begin{itemize} [noitemsep, leftmargin=15pt]
\item for $x \in N$, $\phi^{-1}(x)$ is a totally disconnected set, and
\item there is  a non-empty discrete set $B_{\phi} \subset N$  such that the
restriction $\phi|_{M \setminus \phi^{-1}(B_{\phi})}$ is an ordinary
unbranched
covering of degree $d$,
\end{itemize}
 \vspace*{-7pt}
is called a \emph{branched covering of degree d over N}
 and it is  denoted  by
$(M,\phi,N,B_{\phi},d)$.  $N$ is \emph{the base surface}, $M$ is
\emph{the covering surface} and $B_{\phi}$ is \emph{the branch point set}.\\
The set $B_{\phi}$ is  defined by
the image of the points in $M$ on which $\phi$ fails to be a local
homeomorphism. Thus each $x \in B_{\phi}$ determines a non-trivial
partition $D_{x}$ of $d$, defined by the local degrees of $\phi$.
 The collection $\mathscr{D}:=\{D_{x}\}_{x \in B_{\phi}}$ is called
 \emph{the branch datum} and its  \emph{total defect} is the positive integer
defined by
$\nu({\mathscr{D}}):=\sum_{x \in B_{\phi}} \nu(D_{x})$. \\
The total defect  satisfies
the \emph{Riemann-Hurwitz formula} (see \cite{EKS}):
\begin{eqnarray}\label{rhf}
\nu(\mathscr{D})=d \chi(N)-\chi(M).
\end{eqnarray}
Associated to $(M,\phi,N,B_{\phi},d)$ we have a permutation group, {\it the monodromy group of $\phi$}, given by the
image of the \emph{Hurwitz's representation}
\begin{eqnarray}\label{Hr}
 \rho_{\phi}: \pi_{1}(N-B_{\phi},z) \lra \Sigma_{d},
 \end{eqnarray}
that sends each class $\alpha \in \pi_{1}(N-B_{\phi},z)$ to a permutation of $\phi^{-1}(z)=\{z_{1},
\dots,z_{d}\}$ defined by the terminal point of the lifting of a loop
in $\alpha$ after fixing the initial point.

\begin{Prop}[Lemma 2.1 c), \cite{EKS}]\label{2.1-EKS}
Let $\mathscr{D}= \{D_1,\dots,D_k\}$ be a collection of partitions of d (repetitions allowed). Then  there exists a  branched covering of degree $d$, $\phi: M \rightarrow \mathbb{R}P^2$,  with $M$ closed and connected and with branch data $\mathscr{D}$ if, and only if  there are permutations $\alpha_i \in D_i$ and $\omega \in \Sigma_d$ such that $\langle \alpha_1,\dots, \alpha_k,\omega | \alpha_1\dots \alpha_k \omega^2=1 \rangle$ acts transitively on $\{1,\dots,d\}.$ \qed
\end{Prop}

We state below a  theorem from \cite {EKS} about the realizability of branched coverings  over 
$\mathbb{R}P^{2}.$   We will use that theorem to describe all realizable branched data of 
branched covering over $\mathbb{R}P^2$ which are primitive (surjective on $\pi_1$). 

\begin{theo}[Theorem 5.1,  \cite{EKS}]\label{realization}
Let $\mathscr{D}$ be a collection of partitions of $d$. Then there is a branched covering $\phi:M\rightarrow \mathbb{R}P^{2}$ of degree $d$, with M connected and with branch data $\mathscr{D}$ if and only if 
\begin{eqnarray}\label{hcpp}
d-1\leq \nu(\mathscr{D})\equiv 0 \pmod{2}.
\end{eqnarray} 
Moreover M can be chosen to be nonorientable.
\qed
\end{theo}

\begin{Cor}\label{real-min}  Every collection of  partitions of $d$ with  $ \nu(\mathscr{D})=d-1$ and $d$  odd is realizable by a primitive   branched covering over $\mathbb{R}P^2$. \qed
\end{Cor}
%\textcolor{blue}{N: a seguir (deixo em verde) analizamos o caso $M=S^2$ e mostramos que o Teorema 1 nao considera esse caso, e o EKS nao deixa isso explicito. MAS nesse caso o recobrimento nao \'e sobrejetor no $\pi_1$. 
%Em principio acharia que podemos tirar para ganhar algumas linhas, mas se queremos comparar com o trabalho de EKS  a observa\c c\~ao pode ser  interessante!}\\

%\textcolor{green}{
\noindent We observe that Theorem \ref{realization} does not hold when
$M=S^2$. In fact given any branched covering $\phi:S^2\rightarrow \mathbb{R}P^{2}$ of degree $d$, by the Riemann-Hurwitz formula, we  have $\nu(\mathscr{D})=d-2$ and does not satisfies the condition $(3)$. There are branched coverings having as base   $\mathbb{R}P^{2}$  and total space $S^2$. For example, take   the universal  covering 
$p: S^2\to  \mathbb{R}P^{2}$. We also have  branched coverings which are not covering. For example consider a 
branched covering $\phi: S^2\to S^2$, which is not the identity, then the composite $p\circ \phi:S^2 \to \mathbb{R}P^2$
is a branched covering which is not a covering. In fact, it is easy to see that all branched covering  $S^2 \to \mathbb{R}P^2$ are  obtained as described above.

 \begin{definition}
A collection of partitions $\mathscr{D}$ of $d$ satisfying $(\ref{hcpp})$ will be called {\it admissible datum on $\mathbb{R}P^2$}.
\end{definition}

\section{Characterization of decomposable minimal defect data on $\mathbb{R}P^2$ }\label{decomp}
Decomposable data on $\mathbb{R}P^2$ are the admissible data 
 realizable by 
 decomposable primitive branched coverings over $\mathbb{R}P^2$.  
  
 \begin{Prop}\label{yo}
A primitive branched covering on a closed surface $N$
 is decomposable if and only if its monodromy
group is imprimitive. 
\end{Prop}
\begin{proof} This result is sated in \cite{BN&GDL1} (Proposition 1.8) in the context $\chi(N)\leq 0$. But the same proof  applies for any closed surface $N$.
\end{proof}

In the following we give some general definitions and results needed to prove a version of  Proposition 2.6 \cite{BN&GDL1} (see Introduction) for the  case we are studying. 
  \begin{definition}[\S 2, \cite{BN&GDL1}]\label{Pp} Let  $D=[d_1,\dots,d_t]$ be a partition of a non-prime integer  $d \in \mathbb{N}$. We will say that $D$ is a product partition if  for a non-trivial factorization $d=uw$ of $d$  there exist a positive integer  $s$, a partition
  $U=[u_{1},\dots,u_{s}]$ of $u$ and a collection of partitions $\mathscr{W}=\{W_{1},\dots,W_{s}\}$ of $w$ such that 
  \begin{equation}\label{pp}
  D=[u_1.W_1, \dots, u_s.W_s],
  \end{equation}
   where $u_i.W_i$ denotes the partition of $u_i.w$ obtained
  by multiplying each component of $W_{i}$ by 
$u_{i}$, for $i=1,\dots,s.$ We write $D=U.\mathscr{W}.$
\end{definition}
\noindent From (\ref{pp}),  a product partition of $d$ is a union of partitions of multiples of a proper divisor of $d$.

  \begin{Ex}\label{4}
  For $u=w=3$, the partition
 $[2,2,2,1,1,1]$ could be expressed as  a product partition  in the following ways:  \\either as $\bigl[\mathit{1}[2,1],\mathit{1}[2,1],\mathit{1}[2,1]\bigr]$ with $U=[\mathit{1,1,1}]$ and $\mathscr{W}=\{[2,1],[2,1],[2,1]\}$, or as $\bigl[\mathit{2}[1,1,1],\mathit{1}[1,1,1]\bigr]$ with $U=[\mathit{2,1}]$ and $\mathscr{W}=\{[1,1,1],[1,1,1]\}$.
  \end{Ex}
\begin{definition} \label{fatorization} 
  A  collection $\mathscr{D}$ of partitions   of $d$   admits an algebraic factorization if we can write it as an union of product partitions associated to  the same non-trivial factorization of $d$. Especifically, if for  $d=uw$
  there exist $t \in \mathbb{N}$, $\mathscr{U}=\{U_{1},\dots,U_{t}\}$ a collection of partitions of $u$
and $\mathscr{W}=\bigcup_{i=1}^{t}\mathscr{W}_{i}$ a union of collection of partitions of $w$  such that  $\mathscr{D} =\{U_{i}.\mathscr{W}_{i}\}_{i=1}^{t}$. We will write $\mathscr{D}=\mathscr{U}.\mathscr{W}$ and call $\mathscr{U}$ and $\mathscr{W}$ as ``the first factor" and ``the second factor", respectively.
\end{definition}

\begin{Prop}[Proposition 2.4, \cite{BN&GDL1}]\label{fator}
For $d=uw$ let $\mathscr{D}, \mathscr{U}, \mathscr{W}$ 
be   collections of partitions of  $d,\, u$, $w$ respectively such that 
$\mathscr{D}=\mathscr{U}.\mathscr{W}$. Then
$\nu(\mathscr{D})=\nu(\mathscr{W})+w \nu(\mathscr{U})$. \qed
\end{Prop} 
%\begin{proof}
%Let $\mathscr{D}=\mathscr{U}.\mathscr{W}$ be an algebraic factorization of $\mathscr{D}$. Then there are positive integers $t$
%and $s_{i}$, for $i=1,\dots,t$, such that $\mathscr{U}=\{U_{1},\dots,U_{t}\}$ where $U_{i}$ is a partition of $u $ with $s_{i}$ summands, and 
%$\mathscr{W}=\cup_{i=1}^{t}\mathscr{W}_{i}$ where $\mathscr{W}_{i}$ is a collection of   $s_{i}$ partitions of  $w$. Then $\mathscr{D}=\{U_{i}.\mathscr{W}_{i}\}_{i=1}^{t}$ and
%$\nu(\mathscr{D})=\sum_{i=1}^{t}\nu(U_{i}.\mathscr{W}_{i})=\sum_{i=1}^{t}(uw+\nu(\mathscr{W}_{i})-ws_{i})=\sum_{i=1}^{t}\nu(\mathscr{W}_{i})+\sum_{i=1}^{t}w(u-s_{i})=\nu(\mathscr{W})+w
%\sum_{i=1}^{t}\nu(U_{i})=\nu(\mathscr{W})+w \nu(\mathscr{U})$.
%\end{proof}

\noindent Next proposition  characterize the  decomposable minimal defect data on $\mathbb{R}P^2$.

\begin{Prop}\label{nt-minimal}  An admissible datum $\mathscr{D}$ with minimal defect  on $\mathbb{R}P^2$   is decomposable if, and only if
  there exists an algebraic factorization of
 $\mathscr{D}$ such that its first factor 
is  a non-trivial admissible datum on $\mathbb{R}P^2$ with minimal defect.
\end{Prop}
\begin{proof}
Suppose that $(\mathbb{R}P^2,\phi,\mathbb{R}P^2,B_{\phi},d)$ is a decomposable primitive branched
covering realizing 
$\mathscr{D}$, with $\nu(\mathscr{D})=d-1$. Then  there exist a surface $K$ and coverings 
 $\psi: \mathbb{R}P^2 \rightarrow K$ and $\eta: K \rightarrow \mathbb{R}P^2$,  of degrees $w,u>1$ respectively, such that $\phi=\eta \circ \psi$ and  $d=uw$. From (\ref{rhf}) and (\ref{hcpp}), necessarily  $0 \leq w\chi(K) -1 \equiv 0 \pmod{2}$ then  $K=\mathbb{R}P^2$. Hence
 there is a
 non-empty subset  $B_{\eta}\subset B_{\phi}$ such that 
$(\mathbb{R}P^2,\eta,\mathbb{R}P^2,B_{\eta},u)$ is a primitive branched covering with branch datum
$\mathscr{\widetilde{U}}$
  and minimal defect $\nu(\mathscr{\widetilde{U}})=u-1$,  from (\ref{rhf}).\\
   Note that each
 $x \in B_{\phi}$ determines a partition  of  $u$ (that will be trivial if
$x \in B_{\phi} \setminus B_{\eta}$) and each point in  $\eta^{-1}(x)$ determines a
partition of $w$ (that will be trivial if such a point is not a branch point of
$\psi$). In other words,  $x \in B_{\phi}$ determines a partition 
$U_{x}$ of $u$ and a collection 
$\mathscr{W}_{x}$ of partitions of $w$, such that
$U_{x}.\mathscr{W}_{x}$ is the partition of  $d$ 
that $x$ determines for $\phi$. 
Then $\mathscr{D}=\{U_{x}.\mathscr{W}_{x}\}_{x \in B_{\phi}}$ is
a factorization with an admissible non-trivial first factor $\mathscr{U}=\{U_{x}\}_{x \in B_{\phi}}$ on $\mathbb{R}P^2$ with minimal defect, because $\nu(\mathscr{U})=\nu(\mathscr{\widetilde{U}})$
(the difference between $\mathscr{U}$ and $\widetilde{\mathscr{U}}$ are
just the trivial partitions) and
$B_{\eta} \neq \emptyset$.
\\
Conversely, suppose $d=uw$ and let $\mathscr{D}=\{U_{x}.\mathscr{W}_{x}\}_{x \in B}$ be a factorization of an  admissible datum on $\mathbb{R}P^2$ with minimal defect, whose first 
factor $\mathscr{U}$ is non-trivial and  admissible on $\mathbb{R}P^2$ with minimal defect $u-1$.
By Corollary \ref{real-min} there exists a primitive branched covering
 $(\mathbb{R}P^2,\eta,\mathbb{R}P^2,B_{\eta},u)$ realizing 
$\mathscr{U}$. We  suppose $B_{\eta} \subset B$  and $U_x$ as a trivial partition for  each $x\in B\backslash B_\eta$. By
 Proposition \ref{fator} the second factor $\mathscr{W}$, 
 is admissible on $\mathbb{R}P^2$  with minimal defect $w-1$. Then it is not trivial and
 by Corollary \ref{real-min} there exists 
a primitive branched covering  
 $(\mathbb{R}P^2,\psi,\mathbb{R}P^2,B_{\psi},w)$ realizing it
  as branch datum. Without loss of generality 
we can assume that  $B_{\psi}\subset \eta^{-1}(B)$ and, for each $x\in \eta(B_{\psi})$, we have   $U_x=[u_{x,1},\dots,u_{x,s_x}]$ a partition of $u$, and  
  $\eta^{-1}(x)=\{y_{x,1}, \dots,y_{x,s_x}\}$ such that   $\eta$ has local degree $u_{x,j}$ at $y_{x,j}$. Moreover,  $\psi$ and the point $y_{x,j}$ determine     
   the partition  $\mathscr{W}_{x,j}$ of  $w$, for  $1\leq  j \leq s_x$.  Thus 
$(\mathbb{R}P^2,\eta \psi, \mathbb{R}P^2, B_{\eta}  \cup \eta(B_{\psi}),$ $d=uw)$ is a decomposable primitive branched covering with branch datum
 $\mathscr{D}$. 
\end{proof}
\begin{Cor}\label{U1} Let $\mathscr{D} =\{U_{1}.\mathscr{W}_{1},U_{2}.\mathscr{W}_{2}\}$ be an algebraic factorization of  an admissible datum on $\mathbb{R}P^2$ such that $\nu(\mathscr{D})=d-1$. If   $\mathscr{U}=\{U_1,U_2\}$ is  admissible on $\mathbb{R}P^2$ and 
   $U_2=[1,\dots,1]$ then $U_1=[u].$
\end{Cor}
\begin{proof}
From Proposition \ref{nt-minimal}, 
 $\mathscr{D}$ is realizable by a decomposable branched covering.  Since $\nu(D_1)+\nu(D_2)=d-1$, then (\ref{rhf}) implies that  the map is of the form  $\mathbb{R}P^2\rightarrow\mathbb{R}P^2$  and it is a composition of a pair of maps of the same type.  One of them realizing $\mathscr{U}$ as branch datum, with  $\nu(\mathscr{U})=\nu(U_1)+ \nu(U_2)=u-1$ and $ \nu(U_2)=0$. Hence $\nu(U_1)=u-1$ and  $U_1=[u]$.
  \end{proof}

\begin{Rm}\label{blocos} From Proposition \ref{yo} and Proposition  \ref{nt-minimal} 
there is a relation  between a decomposable realization $(M, \phi, N, B,d)$ of $\mathscr{D}$ and an algebraic factorization of $\mathscr{D}$ with  admissible first factor. In fact, an algebraic factorization $\mathscr{D}=\{U_{x}.\mathscr{W}_{x}\}_{x \in B}$ provides a factorization $d=uw$ and  describes the structure of  a system of  blocks  for  the group defined by the Hurwitz's representation 
 $\rho_{\phi}$ (see (\ref{Hr}) and Proposition \ref{2.1-EKS}) associated to the realization. We illustrate this with an  example:
 
  \vspace{3pt}
\noindent Consider $d=9$ and $ \mathscr{D}=\{[6,1,1,1],[2,2,2,1,1,1] \}.$
 A realization of $\mathscr{D}$ is characterized by the definition of a transitive permutation group $G:=\langle \alpha, \beta, \gamma | \alpha \beta=\gamma^{-2} \rangle$  such that  $ \alpha \in  D_1=[6,1,1,1]$ and  $\beta \in  D_2=[2,2,2,1,1,1]$ (see Proposition \ref{2.1-EKS}). \\
If $ \alpha=(1\;2\;3\;4\;5\;6)(7)(8)(9)$, the factorization $[2[3],1[1,1,1]]$ of $D_1$ describes three blocks of size three,  preserved by $ \alpha$:  

 \vspace{1pt} 
\noindent $\bullet $  two blocks given and preserved by the support of the $6$-cycle: $ \{1,3,5\},\, \{2,4,6\}$ and one block given by  ${\rm Fix}\, \alpha$:
 $ \{7,8,9\}$. 
 
 \vspace{1pt}
\noindent  If we want $G$  imprimitive, it is necessary to define $\beta$ in such a way  that the  blocks preserved by $\alpha $ are also preserved  by $ \beta .$ Then  the factorization $[2[1,1,1],1[1,1,1]]$ of $D_2$ suggests: 

 \vspace{1pt}
\noindent $\bullet \textrm{ two blocks preserved by the  set of transpositions, one block from ${\rm Fix}\,\beta$.}$

 \vspace{2pt}
\noindent   Hence,  regarding the blocks for $\alpha$, we define $\beta=(1\;7)(3\;8)(5\;9)(2)(4)(6)$. 
\end{Rm}
\begin{Rm}\label{PL}
The results above are closely  related to some results in \cite{ACDHL}. To see this, it is necessary to introduce some of their definitions:
\begin{enumerate}[noitemsep, leftmargin=15pt]
\item Let $k$ be a positive integer, and $(w_1, \dots, w_l)$ a partition of w. Then the partition $(kw_1, \dots , kw_l)$ of $kw$ is said to be an ic-partition (for ``imprimitive cycle'') of type $(k, w)$.
\item For a partition $P$ with $r$ parts, a clustering of $P$ is a partition of the set of parts of $P$ into parts (called clusters) $P_1, \dots, P_t$ (each of which is a partition).
\item An i-partition (for ``imprimitive'') of type $(u, w)$ is a partition of $uw$ which has a clustering into clusters which are ic-partitions of $u_iw$ of type $(u_i,w)$ for $i = 1,\dots,r$, where $(u_1,\dots,u_r)$ is a partition of $u$.
\item The i-type of a partition is the set of pairs $(u,w)$ for which the partition has an
i-partition of type $(u, w)$. The i-type of a permutation is the i-type of its cycle partition. A permutation is primitive if and only if its i-type is empty.
\end{enumerate}
\end{Rm}
\noindent Notice that the  definition of ``a i-partition of type $(u,w)$" above is equivalent to the definition of  ``a product partition  for   $d=uw$" in Definition \ref{Pp} (\S 2, \cite{BN&GDL1}). Hence Proposition
\ref{nt-minimal}  in terms of \cite{ACDHL} establishes a kind of generalization of Theorem 2.5 of \cite{ACDHL}:
\begin{Cor} Let $\mathscr{D}=\{D_1,\dots,D_n\}$ be
a collection  of  i-partition of $d$ of type $(u,w)$, with $\nu(\mathscr{D})=d-1$. Then there exists $\alpha_i \in D_i$, for $i=1, \dots,n,$ such that  $\langle \alpha_1, \dots, \alpha_n \rangle$ is  imprimitive with $u$ blocks of cardinality $w$, if and only if $\nu(\mathscr{U})=u-1$, where $\mathscr{U}$ is the collection  of partitions  of $u$ defined by the clusterings in $\mathscr{D}$. \qed
\end{Cor}

\begin{Ex}
Let $A=B=[3,2,2,1,1]$. Then any transitive permutation group $\langle \alpha, \beta \rangle$, for $\alpha\in A,\; \beta \in B$, is primitive. This is because the i-type of $[3,2,2,1]$ is  $\{(3,3)\}$ with unique  clustering given by $1[3],1[2,1], 1[2,1]$ that defines the set $\mathscr{U}=\{[1,1,1],[1,1,1]\}$ with $\nu(\mathscr{U})=0 \neq 2$. In other words, by Remark \ref{blocos},  the only possible system of $3$ blocks of length $3$ has to be given by: one block defined by the $3$-cycle and two blocks defined by a transposition and a fixed element. Thus, if $\alpha \in A$ there is no $\beta \in B$ that can take care  of  transitivity and imprimitivity, simultaneously. 
\end{Ex}

\section{Branched coverings with  at most $2$  branched points}\label{leq2}
\subsection{One branched point}
Considering branched coverings with branching locus in  one point, since we are assuming 
$\nu(\mathscr{D})=d-1$ with $d$ odd,  the admissible branched datum has to be $\mathscr{D}=\{[d]\}$. In \cite{BN&GDL2} this case is solved: 

\vspace{2pt}
\noindent {\it Example 3 \cite{BN&GDL2}: \label{1pto} There is only one branched covering of degree $d$, for $d$ odd,   over the projective plane  with one branched point 
 and $\nu(\mathscr{D})=d-1$. Further, this branched covering
$(\mathbb{R}P^{2},\phi,\mathbb{R}P^{2},\{x\},d)$ is decomposable if, and only if
  $d$ is non-prime.}

\vspace{5pt}
\noindent  The ``if and only if" part of the example above can be proved alternatively as follows:\\ 
 Let $\mathscr{D}=\{D\}$ be the  branch datum. Since $\nu(\mathscr{D})=d-1$ then $D=[d]$ and any non-trivial factorization   of $d$, $d=uw$,  defines an algebraic factorization of $\mathscr{D}$ by admissible factors $\mathscr{U}=[u]$ and $\mathscr{W}=[w]$. Then, by Proposition \ref{nt-minimal}, $\mathscr{D}$ is decomposable. 

\subsection{Two  branched points }\label{2points}
\vspace{5pt}
\noindent Given a partition $D=[k_1,\dots,k_r]$ of $d \in \mathbb{Z}^+$, we set up that $\gcd (D)=\gcd\{k_1,\dots,k_r\}.$  The purpose  of this section is to prove:

\begin{theo}\label{iff2pt}
Let $\mathscr{D}=\{D_{1},D_{2}\}$  be a branch datum such that $\nu(\mathscr{D})=d-1$. Then $\mathscr{D}$ is realizable by an indecomposable branched covering if and only if  $\gcd(D_{i})=1$ for $i=1,2$ and  $\mathscr{D}\neq \{[2,\dots,2,1],[2,\dots,2,1]\} $.
\end{theo}

\noindent Let us start with some generalities.
Let $d \in \mathbb{Z}^{+}$ be an odd integer and  let  $\mathscr{D}=\{D_{1},D_{2}\}$ be  a collection of two partitions of $d$ with $\nu(\mathscr{D})=d-1$. From now on, we will assume:
 \begin{equation}\label{Dis}  \nu(D_{1})\geq \nu(D_{2})
  \textrm{ and we have  } D_{i} \neq [d] \textrm{ for } i=1,2 \;(\textrm{otherwise } \nu(\mathscr{D})>d-1).
 \end{equation}
 Let us set up the notation: 

\vspace{-5pt}
\begin{equation}\label{D1D2}
\begin{array}{ll}
D_1=[c_1,c_2,\dots,c_r,1,\dots,1 ],  \textrm{ with }  c_{i-1}\geq c_i>1 \textrm{ and }  \ell_1 \textrm{ ones,  and }\\
     D_2=[e_1,e_2,\dots,e_s,1,\dots,1], \textrm{ with } e_{j-1} \geq e_j>1 \textrm{ and } \ell_2 \textrm{ ones.} 
     \end{array}
     \end{equation}
     \begin{Rm}\label{numberrestr} From (\ref{Dis}) and (\ref{D1D2}) hold:
     \vspace{-5pt}
     \begin{itemize} [noitemsep, leftmargin=15pt]
     \item[a)] $r+\ell_1\leq s+\ell_2$ and $r+\ell_1+s+\ell_2=d+1$ because  $\nu(D_{1})\geq \nu(D_{2})$and $\nu(\mathscr{D})=d-1$;
     \item[b)] $\ell_2\geq 1$ because $\nu({D_2})\leq (d-1)/2$; 
     \item[c)] if $\ell_2=1$, since $\nu({D_2})\leq (d-1)/2$ and equivalently  $s+\ell_2 \geq (d+1)/2$, then $D_2=[2,\dots,2,1]$ and   $\nu(D_2)=(d-1)/2=\nu(D_1)$.
        \item[d)] $r+ \ell_1=\nu(D_2)+1$ and $s+\ell_2=\nu(D_1)+1$, this is straightforward from the equality  $\nu(\mathscr{D})=d-1$. 
     \end{itemize}
     \end{Rm}

\begin{Lemma}\label{unos} Let  $\mathscr{D}=\{D_{1},D_{2}\}$ be as in  (\ref{Dis}) and (\ref{D1D2}).
Then $\ell_1\leq c_1+\dots+c_r-2r+1$  and   $\ell_1= c_1+\dots+c_r-2r+1$  if and only   if $\nu({D_1})= (d-1)/2=\nu({D_2}).$
 \end{Lemma}
\begin{proof}  Call $k=c_1+\dots+c_r$. By hypothesis we have  $\nu({D_1})\geq (d-1)/2$. Equivalently  $d-\ell_1-r\geq (d-1)/2$ and     $\ell_1 +k=d$. Hence $\ell_1+k-\ell_1-r\geq (\ell_1+k-1)/2$ which is equivalent to 
 $\ell_1\leq k-2r+1$. Also follows that   $\ell_1= k-2r+1$ if and only if $\nu({D_1})=(d-1)/2$. 
 \end{proof}
\begin{Lemma}\label{s<a} Let $d \in \mathbb{Z}^{+}$ be an odd integer and $\mathscr{D}=\{D_{1},D_{2}\}$
 with $\nu(\mathscr{D})=d-1.$ If $\ell_1=0$ then  $c_i \leq \ell_2$ for all summand $c_i$ of $D_1$. 
\end{Lemma}
\begin{proof}
 It is enough to show that $c_1 \leq \ell_2$. Since $\ell_1=0$, from Remark \ref{numberrestr} d) we have:
  \begin{equation}\label{e}
  \begin{array}{ll}
 e_1+(e_2-1)+\dots+(e_s-1)=r, \textrm{ and}\\
 c_1+(c_2-1)+\dots+(c_r-1)=s+\ell_2.
 \end{array}
 \end{equation}
Thus we have $s \leq r-1$ and since  $c_i \geq 2$, for $i=1,\dots,r,$ 
  then (\ref{e}) implies $c_1\leq \ell_2$ and therefore $c_i\leq \ell_2$, for $i=1,\dots,r.$
  \end{proof}

 By Proposition \ref{2.1-EKS}, the realization of $\mathscr{D}=\{D_1,D_2\}$ is associated with the construction of a transitive permutation group  $G=\langle \alpha, \beta, \gamma | \gamma^2= \alpha \beta \rangle$ such that $\alpha \in D_1,\; \beta \in D_2$. In this sense, we look at the possibilities.
 
\begin{Lemma}\label{d-ciclo}Let $\mathscr{D}=\{D_1,D_2\}$  be a collection of partitions of an odd integer $d$ with $\nu(\mathscr{D})=d-1$. Let $\alpha, \beta$ be permutations such that $\alpha \in D_1,\; \beta \in D_2$. Then  
  $\langle \alpha, \beta \rangle$ is transitive if and only if  $\alpha \beta$ is a $d$-cycle, and so a square. 
\end{Lemma}
\begin{proof}
If $\alpha \beta$ is a $d$-cycle, it is clear that $\langle \alpha, \beta \rangle$ is transitive.\\
For the converse, using the notation in  (\ref{D1D2}), we have $\nu(\beta)=\nu(D_2)=\sum_{i=1}^{s}(e_i-1)$, see subsection \S \ref{PG}. 
 Let $\beta_i$ be a $e_i$-cycle for $i=1,\dots, s.$
Notice that ${\rm Supp}\,\beta$ and $\nu(\beta)$ are defined by these $\beta_i$'s. \\
  By Remark \ref{numberrestr} d) we have that  $e_1+\sum_{i=2}^{s}(e_i-1)=r+\ell_1$, where $r+\ell_1$ is the number of cycles of $\alpha$. Then  if $\langle \alpha, \beta\rangle$ is transitive, the last equality implies that $\beta_{1}$  connects $e_1$ cycles of $\alpha$ and  $\beta_{i}$ connects $e_{i}$ disjoint cycles of $\alpha \beta_{1} \dots \beta_{i-1}$,
 for $i=2,\dots,s$.  Thus the number of cycles of 
$\alpha\beta_{1}$ is $r+\ell_1-e_1+1$ and 
the number of cycles of 
$\alpha \beta=\alpha\beta_{1}\dots\beta_s$ is $r+\ell_1-\sum_{j=1}^s (e_j-1)=1.$ Hence   $\alpha \beta$ is a $d$-cycle, and so a square. 
\end{proof}

It is straightforward from the proof of Lemma \ref{d-ciclo} that:
\begin{Cor}\label{b_i} With the hypothesis   and the notation in Lemma \ref{d-ciclo}, the group   $\langle \alpha, \beta \rangle$
 is transitive if and only if  each $e_i$-cycle of $\beta$  has exactly one element in common with $e_i$ different  cycles of  $\alpha$, 
 for $i=1,\dots,s$ and all cycles of $\alpha$ are attained. \qed
\end{Cor}

 \subsubsection{Co-existence of decomposable and indecomposable realizations}\label{DI}
 Recall that  Proposition \ref{nt-minimal} implies that a datum   does not admit an algebraic factorization  if and only if  it does not admit a decomposable realization. Since  decomposable and indecomposable realizations could coexist, let us focus on the question:
 
\vspace{3pt}
\noindent   {\it Which data, in the class of decomposable ones,   are also realizable by indecomposable branched coverings?  }\\

\vspace{-10pt}
\noindent Let us start by studying the case   $\mathscr{D}\neq\{[2,\dots,2,1],[2,\dots,2,1]\}$, with  $\nu(\mathscr{D})=d-1$ and satisfying (\ref{Dis}) and (\ref{D1D2}), i.e.:
\vspace{-5pt}
\begin{equation*}
\begin{array}{ll}
D_{i} \neq [d] \textrm{ for } i=1,2 
 \textrm{ and } \nu(D_{1})\geq \nu(D_{2}),\\
D_1=[c_1,c_2,\dots,c_r,1,\dots,1 ],  \textrm{ with }  c_{i-1}\geq c_i>1 \textrm{ and }  \ell_1 \textrm{ ones,  and }\\
     D_2=[e_1,e_2,\dots,e_s,1,\dots,1], \textrm{ with } e_{j-1} \geq e_j>1 \textrm{ and } \ell_2 \textrm{ ones.} 
     \end{array}
     \end{equation*}

\begin{Lemma}\label{3cycles0} 
Let $d \in \mathbb{Z}^{+}$ be an odd  integer, $\mathscr{D}=\{D_{1},D_{2}\}$ a collection of partitions of $d$ satisfying (\ref{Dis}) and (\ref{D1D2})
    with $\nu(\mathscr{D})=d-1$ and $D_1\ne [2,\dots,2,1]$. Then:
    \vspace{-7pt}
    \begin{enumerate}[noitemsep, leftmargin=15pt]
  \item  If   $\ell_1=0$    and $D_1$  contains summands $x ,\, y$ such that $x> y$  then there exist $\alpha \in D_1$, 
    $\beta \in D_2$ such that  $\alpha \beta \in [d-2y,y,y]$;
 \item If $\ell_1>0$  then 
   $D_1$  contains a  summand $x >2$,  and 
    there exist $\alpha \in D_1$, $\beta \in D_2$ such that $\alpha \beta \in [d-2,1,1]$.
    \end{enumerate}
      \vspace{-3pt}
    Further,  there exists a permutation $\omega$ such that $\langle \alpha, \beta, \omega| \omega^2=\alpha \beta \rangle$ is a transitive permutation group.
    \end{Lemma}

   \begin{proof} 

\noindent $1)$ If  $\ell_1=0$ and $D_1$  contains summands $x ,\, y$ such that $x> y$: \\
 Since $c_{j}>1$, for $j=1,\dots,r$ then $x>y\geq 2$. 
Relabeling the summands of $D_{1}$ we write: \\
\centerline{$D_{1}=[x,\kappa_1,\dots,\kappa_{r-2},y],$}\\
 with $\kappa_{1} \geq \dots \geq  \kappa_{r-2}$. 
Let $\alpha=\gamma_{\rm x}\,\alpha_1\,\dots\,\alpha_{r-2} \,\gamma_{\rm y} \in D_{1}$ be an arbitrary permutation with the cyclic structure given by $D_{1}$,  
where $\gamma_{\rm x}$ is a $x$-cycle, $\gamma_{\rm y}$ is a $y$-cycle and $\alpha_i$ is    a $\kappa_{i}$-cycle for $i=1,\dots,r-2.$
We will define $\beta=\beta_1\dots\beta_s \in D_2=[e_{1},e_{2},\dots,e_{s},1,\dots,1]$, with  $\beta_i$  a $e_i$-cycle,  such that:  

\vspace{3pt}
\noindent i)  
If   $s=1$ then  $\beta=\beta_1$ and  $e_1=r$, because $\nu(\mathscr{D})=d-1$.  Suppose $\gamma_{\rm x}=(z_{{\rm x}_{1}} \dots z_{{\rm x}_{x}})$ and   define:
\begin{eqnarray*}
  \beta=
  \begin{cases}
  ( z_{{\rm x}_{1}}^{(\gamma_{\rm x})^y}\; z_{{\rm x}_{1}}), \textrm{ if $r=2$}\\
( z_{{\rm x}_{1}}^{(\gamma_{\rm x})^y}\; z_{{\rm x}_{1}}\; z_{1_1} \dots\;z_{(e_{1}-2)_1}), \textrm{if $r>2$}
\end{cases}
\end{eqnarray*}
for 
 $z_{j_1} \in Supp \; \alpha_{j}$, 
  $j=1,\dots,e_{1}-2$.   Observe that ${\rm Supp}\,\gamma_{\rm y} \subset {\rm Fix}\, \beta$.
  Then we have
 \begin{eqnarray*}
\alpha \beta=
\begin{cases}
(\underbrace{z_{{\rm x}_{1}}\dots z_{{\rm x}_{1}}^{(\gamma_{\rm x})^{y-1}}}_{y})(z_{{\rm x}_{1}}^{(\gamma_{\rm x})^{y}}\dots z_{{\rm x}_{x}})\gamma_{\rm y}, \textrm{ if $r=2$},\\
(\underbrace{z_{{\rm x}_{1}}\dots z_{{\rm x}_{1}}^{(\gamma_{\rm x})^{y-1}}}_{y})(z_{{\rm x}_{1}}^{(\gamma_{\rm x})^{y}}\dots z_{{\rm x}_{x}}\; \iorarrow[\alpha_1]{z_{1_{1}}} \, \dots\, \iorarrow[\alpha_{e_1-2}]{z_{(e_{1}-2)_1}} ) \gamma_{\rm y}, \textrm{ if $r>2$}
\end{cases}
\end{eqnarray*}
where   $\iorarrow[\alpha_{\bullet}]{z_{\bullet_1}}$ denote  the sequence  determined by the cycle $\alpha_{\bullet}$ 
starting in 
$z_{\bullet_1}$. Hence $\alpha \beta$ is  a product of three disjoint cycles, where two of them have length $y$.

\vspace{2pt}
\noindent    ii)
 If $s>1$, let us first define  
     $\beta_i$, for $i=1,\dots,s-1$.  The first equality in  (\ref{e})  suggests to define  $s-1$  subsets $\Lambda_1,\dots, \Lambda_{s-1}$  of $\{x,\kappa_1,\dots,\kappa_{r-2}\}$ (the list of  summands of $D_1$ without $y$)  such that $|\Lambda_j|=e_j$, $|\Lambda_i \cap \Lambda_j|\leq 1$  if  $i \neq j$, for $i,j=1, \dots, s-1$.
    Set:\\
 \vspace{4pt}   
 \centerline{$\Lambda_1=\{x,\kappa_1,\dots,\kappa_{e_1-1}\}$.}\\
 The cycles of $\alpha$ associated to the elements in  $ \Lambda_1$ are $\{\gamma_{\rm x}, \alpha_{1}, \dots, \alpha_{{e_1-1}}\}$ and
  if  $\beta_{1}=(z_{{\rm x}_1}\, z_{1_1} \dots z_{(e_{1}-1)_1})$
 with  $z_{{\rm x}_1}\in  Supp \; \gamma_{\rm x}$, 
 $z_{j_1} \in Supp \; \alpha_{j}$, 
 for $j=1,\dots,e_{1}-1$, 
then $\Gamma:=\gamma_{\rm x}\, \alpha_{1}\, \dots \, \alpha_{{e_1-1}} \beta_1$ is a $(x+\sum_{i=1}^{e_1-1} \kappa_i)$-cycle, from proof of Lemma \ref{d-ciclo}. In fact, if $\gamma_{\rm x}=(z_{{\rm x}_1}\;z_{{\rm x}_2} \dots z_{{\rm x}_x})$ and $\alpha_{j}=(z_{j_1}\;z_{j_2} \dots z_{j_{\kappa_j}})$ for $j=1,\dots, e_1-1$ then 
   $$\Gamma=\Big(\underbrace{z_{{\rm x}_1}\;z_{{\rm x}_2} \dots   z_{{\rm x}_x}}_{\iorarrow[\gamma_{\rm x}]{z_{{\rm x}_{1}}}}\; \underbrace{z_{1_1} \; z_{1_2} \dots z_{1_{\kappa_1}}}_{\iorarrow[\alpha_1]{z_{1_{1}}}} \; \underbrace{z_{2_1} \; z_{2_2} \dots z_{2_{\kappa_2}}}_{\iorarrow[\alpha_2]{z_{2_{1}}}}\dots \underbrace{z_{{(e_1-1)}_1} \; z_{{(e_1-1)}_2} \dots z_{{(e_1-1)}_{\kappa_{(e_1-1)}}}}_{\iorarrow[\alpha_{(e_1-1)}]{z_{{(e_1-1)}_{1}}}}\Big).$$
   Setting $*=z_{{\rm x}_{1}}^{{(\gamma_{\rm x})^{x-(y-1)}}}$ and $\star=z_{1_{2}}$ we have:
   \begin{small}
   \begin{equation}\label{Gama}
   \hspace*{-0.95cm}
   \begin{array}{ll}
   \Gamma=\Big(z_{{\rm x}_1}\;\underbrace{z_{{\rm x}_1}^{{\gamma_{\rm x}}} \dots z_{{\rm x}_{1}}^{{(\gamma_{\rm x})^{x-y}}}}_{\geq 1} \; *\; \underbrace{z_{{\rm x}_{1}}^{{(\gamma_{\rm x})^{x-(y-2)}}} \dots  z_{{\rm x}_{1}}^{{(\gamma_{\rm x})^{x-1}}}\; z_{1_1} \; \star}_{y}  \dots  z_{1_{\kappa_1}} 
   \dots z_{{(e_1-1)}_1}   \dots z_{{(e_1-1)}_{\kappa_{e_1-1}}}\Big )
   \end{array}
   \end{equation}
   \end{small}
   where the sequence $z_{{\rm x}_{1}}^{{(\gamma_{\rm x})^{x-(y-2)}}}  \dots  z_{{\rm x}_{1}}^{{(\gamma_{\rm x})^{x-1}}}$ will be empty if $*=z_{{\rm x}_{1}}^{{(\gamma_{\rm x})^{x-1}}}$, or equivalently if $y=2$.  Hence
   $\star=*^{\Gamma^{y}}$ with
   $\{*,  \star \}\subset Supp\; \Gamma\setminus Supp\; \beta_{1}.$
   
\noindent    If $s>2$,  for  $j=1,\dots,s-1$ we define $p_j:=\sum_{l=1}^j e_l$ and 
for $2\leq j+1\leq s-1$ let\\
 \vspace{4pt}   
    \centerline{$ \Lambda_{j+1}=\{\kappa_{p_j-1}, \kappa_{p_j}, \dots, \kappa_{p_{j+1}-1 }\}.$}

  \noindent       Let $\beta_i$ be  an $e_i$-cycle  for $i=1,\dots,s-1$
     such that   ${\rm Supp}\,\beta_i$ intersects exactly once each cycle   of $\alpha$    associated to the elements in $\Lambda_i$ in such a way that  ${\rm Supp}\,\beta_i \cap {\rm Supp}\,\beta_j=\emptyset$, for $i \neq j$. Since $\kappa_j \geq 2$ for $j=1,\dots, r-2$,  we  have enough ``not used" elements in the complement of $\{*, *^{\Gamma},*^{\Gamma^{2}},\dots, \star \}$ (see (\ref{Gama})) to define $\beta_i$, for $i=2,\dots, s-1.$ 
     
     Let us define $q:=e_1+(e_2-1)+\dots+(e_{s-1}-1)$.   
By the proof of Lemma \ref{d-ciclo}, we have that:
 $ \Gamma \, \alpha_{e_1} \dots \alpha_{q-1} \,\beta_1\dots\beta_{s-1}$
is a $(x+\sum_{j=1}^{q-1} \kappa_j)$-cycle, with
 $\sum_{i=1}^{s-1}e_i$ elements in ${\rm Supp}\, \beta_1 \dots \beta_{s-1}.$ 

\noindent To define $\beta_s$ notice, from the first equality in (\ref{e}), that:
\begin{equation*}
\alpha \beta_{1}\dots\beta_{s-1}=
\begin{cases}
 \Gamma \gamma_{\rm y},& \textrm{ if $e_s=2$,}\\
\underbrace{\Gamma \, \alpha_{e_1} \dots \alpha_{q-1} \,\beta_1\dots\beta_{s-1}}_{\textrm{ $(x+\sum_{j=1}^{q-1} \kappa_j)$-cycle}} \alpha_{q}\;\dots\;\alpha_{r-2}\; \gamma_{\rm y},& \textrm{ if $e_s>2$.}
\end{cases}
 \end{equation*}
If $e_s=2$, we define $\beta_s=(\ast\, \star)$, see (\ref{Gama}).
       Then it follows that $\alpha\beta=\alpha \beta_1\cdots \beta_{s-1}\beta_s=\Gamma \gamma_{\rm y}\beta_s$
      which is a product of 2 cycles of length $y$ and one of length $d-2y$, because $\beta_s$ breaks  the $(d-y)$-cycle $\Gamma $ in two cycles, one of length $y$ and one of length $d-2y$.

  \noindent  If $e_s>2$,
  define $\flat:=(\ast  \; z_{s_{q}}\dots z_{s_{(r-2)}})$ with
  $z_{s_{i}} \in Supp \;\alpha_{i} $, for $i=q,\dots,r-2$.
Thus,  \\
\centerline{$\alpha \beta_{1} \dots \beta_{s-1}\;\flat =\Delta \gamma_{\rm y},$}\\
 where $\Delta$  is a $(d-y)$-cycle   by the proof of  Lemma \ref{d-ciclo} and Corollary \ref{b_i}.  Hence, we define $\beta_{s}:=\flat\; (*\;\star)$ and
 we have:\\
\centerline{$\alpha\beta= \alpha \beta_{1}\dots \beta_{s-1}\beta_{s}=(\star \; \star^{\Delta} \dots \star^{\Delta^{d-2y-1}})
(* \; *^{\Delta}
\dots *^{(\Delta)^{y-1}}) \gamma_{\rm y},$}\\
%\begin{eqnarray}\label{lb}
%\alpha\beta= \alpha \beta_{1}\dots \beta_{s-1}\beta_{s}=(\star \; \star^{\Delta} \dots *^{(\Delta)^{-1}})
%(* \; *^{\Delta}
%\dots *^{(\Delta)^{y-1}}) \gamma_{y},
%\end{eqnarray}
as a product of three disjoint cycles, where two of them have length $y$.

\noindent $2)$ If $\ell_1>0$: since $\nu(D_1) \geq (d-1)/2$, then  $D_1$  contains a  summand $x >2$. In this case we choose  $y=1$.
Notice that $d=\ell_1+ \sum_{j=1}^r c_j= \sum_{k=1}^{s-1}e_k + e_s+ \ell_2,$ and since  $\ell_2>0$ then:  \\
\centerline{
$(\ell_1-1)+\sum_{j=1}^r c_j- \sum_{k=1}^{s-1}e_k=e_s+(\ell_2-1) \geq 2$.}\\
 This, together with $x \geq 3 $, guarantees that
   we can define $\beta$   in an analogous way to the previous case but choosing $\star=\ast^{\gamma_{\rm x}}$. Then   $\{\ast, \ast^{\gamma_{\rm x}}\} \subset {\rm Supp}\, \gamma_{\rm x} \cap {\rm Fix}\, (\beta_1 \dots \beta_{s-1})$
and $\alpha \beta \in [d-2,1,1]$.  
 
%\textcolor{blue}{
For the further part, from the constructions above, we note that whatever the case, the group $\langle \alpha, \beta \rangle$ defines two orbits on $\{1,\dots,d\}$: ${\rm Supp}\,\gamma_{\rm y}$ and $\{1,\dots,d\}\setminus {\rm Supp}\,\gamma_{\rm y} $. \\To define $\omega$ we use the structure of   $\alpha \beta= \pi_1\pi_2\gamma_{\rm y} \in [d-2y,y,y]$, where $\pi_1$ is  a $(d-2y)$-cycle and $\pi_2,\gamma_{\rm y}$ are $y$-cycles.  Since $d$ is odd,  $d-2y$ is odd and $\pi_1^{(d-2y+1)/2}$ is a $(d-2y)$-cycle. With $\pi_2$ and $\gamma_{\rm y}$ we define the $2y$-cycle $\Delta$ such that $\Delta^2=\pi_2\gamma_{\rm y}$.  Set $\omega :=\pi_1^{(d-2y+1)/2} \Delta$ then $\omega^2=\alpha \beta.$ Since $\omega$ links the two orbits mentioned above,  the permutation group $\langle \alpha, \beta, \omega| \omega^2=\alpha \beta \rangle$ is transitive.
%}

  Hence the  lemma is proved.
\end{proof}
%\textcolor{green}{!!!D!!!March16 Let us discuss if what is written here in blue reflect what we want to say}
 From a pair of partitions, in a 
large family of pairs of partitions,  the result above provide    
% for ,  from a large family of pairs of permutations, 
a permutation with a cyclic structure of the type $[d-2y,y,y]$, that is:\\
$\bullet$  the square of a permutation,\\
$\bullet$ primitive  if $\gcd(y,d)$=1 for $d$ odd, see Remark \ref{GEN1}. 
\begin{Prop}\label{3cycles-primitive}
Let $d \in \mathbb{Z}^{+}$ be an odd  integer and $\mathscr{D}=\{D_{1},D_{2}\}$ be  a branch datum
 with $\nu(\mathscr{D})=d-1$ and $\nu(D_1)\geq \nu(D_2)$. If   $D_1$  contains summands $x ,\, y$ such that $x>{\rm max}\{y,2\}$ and $\gcd(y,d)=1$ then  there exist $\alpha \in D_1, \beta \in D_2$ such that the square root  $\sqrt{\alpha \beta}$ exists and  $\langle \alpha, \beta, \sqrt{\alpha \beta} \rangle$ is a primitive permutation group.
 \end{Prop}
\begin{proof}
From Lemma \ref{3cycles0}, there exist $\alpha \in D_1$, $\beta \in D_2$ such that $\alpha \beta \in [d-2y,y,y]$. Since $gcd(y,d)=1$ then, from Lemma \ref{gcd}, the group $\langle \alpha, \beta, \sqrt{\alpha \beta}\rangle$ is primitive.
\end{proof}

%{\bf !!!D!!!Sep/22 Versao revisada com apenas  pequenos detalhes de redacao em 3 ponots. Sem comentario em negrito }

%\textcolor{green}{
%\noindent {\bf Notation:} Given a permutation $\alpha \in D$, we set up that  $\gcd(\alpha):=\gcd(D)$. Remember that
%if $D=[k_1,\dots,k_r]$ then  $\gcd (D)=\gcd\{k_1,\dots,k_r\}.$  \\
Notice that the data $\mathscr{D}=\{D_{1},D_{2}\}$
 with $\nu(\mathscr{D})=d-1$ that do not satisfy the hypothesis of  Proposition \ref{3cycles-primitive} are  either:
% {\bf  !!!D!!!Sep/18 D\'uvida: Tome a parti\c c\~ao:$\{6,3,3,3,\}, \{3,2,2,1,1,1,1,1,1,1,1\}$. Me parece que o dado  nao satisfaz as hipotese da Prop. 2.9 pois so temos 
% para $D_1$ x=6 e y=3 mas $gcd(15,3)\ne 1$. Porem a parti\c c\~ao acima n\~ao satisfies   1) abaixo e  nao satisfies 2 abaixo. Vejamos 3. O item 3 diz que 
% devemos olhar os pares $(x,y)$ em    $D_2$. Portanto x=3 e y=1 satisfaz $x>max\{2,y\}$ mas $gcd(3,1)=1$ logo este par nao satisfaz 3. Pergunto 
% devemos  escrever ao inves de "in $D_1$ or $D_2$," apenas  "in $D_1$,"?  Se sim veja o comentario depois dos itens.}
% \textcolor{orange}{
\begin{enumerate}[noitemsep, leftmargin=15pt]
\item   the summands of $D_1$  are all less than $3$. In this case $\nu(D_1)\leq (d-1)/2$. Since $\nu(\mathscr{D})=d-1$ and $\nu(D_1)\geq \nu(D_2)$ then $\nu(D_1)=(d-1)/2$ and  $D_1=D_2=[2,\dots,2,1],$ or
 \item   $D_1=[k,\dots,k]$  with $k> 2$. In this case  $\gcd(D_1)=k,$ or
 \item   for all  pair of summands $x,y$ with  $x>{\rm max}\{y,2\}$,  in $D_1$,
 % \textcolor{red}{or $D_2$}, 
  we have $\gcd(y,d)\neq 1.$ In this case the summands of $D_1$ are all $\geq 3$.
 \end{enumerate}

% {\bf !!!D!!!Sep/18 Assumindo que na reda\c c\~ao acima cancelemos o $D_2$, me parece que o  item 2 esta contido no item 3. 
% A menos que o item 3 seja substituido pelo item 3':} \textcolor{green}{for all  pair of summands $x,y$ with  $x>{\rm max}\{y,2\}$,  in $D_1$, where the set of such pairs 
% $(x,y)$ is not empty,  we have $gcd(y,d)\neq 1.$ In this case the summands of $D_1$ are all $\geq 3$.
% {\bf !!!D!!! In the next lemma it is written that   $\gamma(D_{i})=1$ for $i=1,2$, but it is not in Lemma 2.7, Lemma 2.8 and Proposition 2.9. 
%Do we need these hypotheses for the statements just mentioned?}
% \textcolor{green}{
\noindent Let us  study   the case where $\gcd(D_1)=\gcd(D_2)=1$, for  $\mathscr{D}\neq\{[2,\dots,2,1],[2,\dots,2,1]\}$.
%\textcolor{blue}{N: A seguir comparo nosso caso com um caso estudado por P. Lopes, s\'o para esclarecer que %estamos considerando uma familia maior. Acho que poderiamos deixar }
%\noindent  \textcolor{red}{???????????????
Notice that the family of partitions $D$ with $\gcd(D)=1$ is larger than the family of relatively prime partitions defined in \cite{Lo}: { \it a partition of $d$ into distinct integers, $[n_1,\dots,n_l]$, is called
a relatively prime partition of $d$ if the parts are mutually prime  i.e., $(n_t,n_{t'})=1$ for each $t \neq t' \in \{1,2,\dots,l\}$}. For these they show that a particular subfamily is realized only by primitive permutations, i.e.  they are not decomposable. In this section,  we are working with partitions   in the complement  of this particular subfamily.

  From Remark \ref{numberrestr} item b) we observe  that we already have $\gcd(D_2)=1$. So the  cases to be considered  will depend on the value 
  $\gcd(D_1)$.

\begin{Lemma}\label{C123} Let $d \in \mathbb{Z}^{+}$ odd  and $\mathscr{D}=\{D_{1},D_{2}\}$  a collection of partitions of $d$ with $\nu(\mathscr{D})=d-1$ and 
$\nu(D_1)\geq \nu(D_2)$. If  $\mathscr{D}\neq \{[2,\dots,2,1],[2,\dots,2,1]\}$,  $\gcd(D_{1})=1$ 
% \textcolor{red}{$\gamma(D_{i})=1$  for $i=1,2$}
 and  every summand of $D_1$ is at least $3$,
then for  $\alpha \in D_1$ there exists $\beta \in D_2$ such that $\langle \alpha, \beta \rangle$ is a primitive permutation group. 
\end{Lemma}
\begin{proof} 
If $\mathscr{D}$ does not admit an algebraic factorization, by Proposition \ref{nt-minimal} and Proposition \ref{yo}, any primitive realization of  $\mathscr{D}$  is indecomposable.
%  for an $\alpha \in D_1$ and $\beta \in D_2$ if $\langle \alpha, \beta \rangle$  is transitive then it is primitive.

Let us assume  that $\mathscr{D}$ admits an algebraic factorization.
%Since $\mathscr{D}$ is decomposable,
% by Proposition \ref{nt1}, there exists an algebraic factorization of
% $\mathscr{D}$ such that its first factor 
%is  a non-trivial admissible datum on $\mathbb{R}P^2$. That is, there exists  a non-trivial factorization of the degree, $d=uv$,  such that $D_1=U_{1}.\mathscr{W}_{1}$ and  $D_2= U_{2}.\mathscr{W}_{2}$, where  $\mathscr{U}=\{U_1,U_2\}$ is a  non-trivial admissible datum on $\mathbb{R}P^2$ (see Definition \ref{fatorization}). Thus by Propostion \ref{fator},
%$u-1=\nu(\mathscr{U})\equiv 0 \pmod{2}$. \\
%%But $D_2=[\underbrace{2,\dots,2}_{s>0},\underbrace{1,\dots,1}_{\ell_2>1}]$ and $d$ is odd then $U_2=[1,\dots,1]$ is trivial and $\nu(\mathscr{U})=\nu(U_1)=u-1$. Hence $U_1=[u]$
%%%and we have $\mathscr{U}\neq \{[1,\dots,1],[1,\dots,1]\}$.
%Suppose  $U_{i}$ with $s_i$ sommands and
%   $\mathscr{W}_{i}$  with $s_i$ partitions of $w$, with
%   $ 2<s_1+s_2<2u.$ The first inequality
% comes from $\gamma(D_1)=\gamma(D_2)=1$ and  the second one comes from $\nu(\mathscr{U})=u-1$.\\
%  Suppose
%  $s_1<u,$ $ \,U_1=[u_{1_1},\dots,u_{1_{s_1}}] $ with $u_{1_1}> 1$,$\, \mathscr{W}_1=\{W_{1_1},\dots,W_{1_{s_1}}\}$ and 
%  $$ \alpha \in U_1.\mathscr{W}_1=\{u_{1_1}.W_{1_1},\dots, u_{1_{s_1}}.W_{1_{s_1}}\}.$$
%   This factorization defines  blocks for $\langle \alpha \rangle$ (see Remark \ref{blocks}): with the elements in the cycles of $\alpha$ corresponding to the sommands in $u_{1_i}.W_i$,  we define $u_{1_i}$ blocks  of  size $w$ preserved by $\alpha$ (see \S \ref{PG}), for $i=1,\dots,s_i$.  \\
 Let $\alpha \in D_1=[c_1, \dots, c_r]$ be a permutation and suppose that $\alpha =\alpha_1\dots \alpha_r$ where $\alpha_i$ is a $c_i$-cycle, for $i=1,\dots,r.$  
 % By the hypothesis we have $|{\rm Supp}\,\alpha_i|\geq 3$, for $i=1,\dots,r.$ \\
  We will define a permutation  $\beta \in D_2=[e_{1},e_{2},\dots,e_{s},1,\dots,1]$ in such a way that:
  \vspace{-7pt}
  \begin{itemize}[noitemsep, leftmargin=15pt]
\item[1)] $\langle \alpha,\beta \rangle$ will be transitive, equivalently  $\alpha \beta$ will be a $d$-cycle by Lemma \ref{d-ciclo}, and 
\item[2)]  $\beta$ will undo  the  possible non trivial blocks of $\langle \alpha \rangle$.
%,   which is known to have  length any proper divisor of $d$.
\end{itemize}
 \vspace{-7pt}
To obtain a $\beta$ satisfying 1),  we use  Corollary \ref{b_i}, i.e., we will define each $e_i$-cycle of $\beta$  with exactly one element in common with $e_i$ different  cycles of  $\alpha$, 
 for $i=1,\dots,s$, and all cycles of $\alpha$ are  attained. To have a $\beta$ satisfying 2) 
 %undo the possible blocks of $\langle \alpha \rangle$ 
 we will take into consideration  the following sequence of facts. 
 
% \textcolor{green}{!!!D!!!March16 At this point of the proof, what is the status of the definition of $\beta$? Has it already definide or suppose it is defined and satisfies what?} %\textcolor{red}{ainda nao, em vermelho acima quis descrever a estrategia para definir $\beta$}

\vspace{3pt} 
\noindent  {\it {\bf Fact 1:} Since $\alpha \beta$ is a $d$-cycle, a system with $u$ blocks for $\langle \alpha ,\beta \rangle$, whether they exist, will be determined   by the cycles of $(\alpha \beta)^u$, for a proper divisor $u$ of $d$. }

%What comes next When I look at the explanation of Fact 1 give me the impression that we are assuming that there is a $\beta$ which satisfies several conditions.
%Is this correct? Could we write something like Fact : Assuming that we have $\beta$ which satisfies. ????? then..... Perhaps from the previous proof we can assume that. Make %sense?  }

\vspace{5pt}
\noindent  On the other hand,  by the hypothesis we have $|{\rm Supp}\,\alpha_i|\geq 3$, for $i=1,\dots,r$ and we can define $\beta$ in such a way  that for $i=1,\dots,r:$ 
\begin{eqnarray}\label{FixB}
 |{\rm Supp}\, \beta \cap {\rm Supp}\, \alpha_i|\leq 2.
%  \;  \textrm{ and } \;
% |{\rm Supp}\, \beta \cap {\rm Supp}\, \alpha_x|=1.
  \end{eqnarray}
If $ |{\rm Supp}\, \beta \cap {\rm Supp}\, \alpha_i|= 2$, we will choose two consecutive elements in $\alpha_i$ for this intersection. % Hence
% \begin{eqnarray}
%\textrm{$|{\rm Supp}\, \alpha_i \cap {\rm Fix} \, \beta| \geq 1$, for $i=1,\dots,a$}.
%\end{eqnarray}
Hence the $d$-cycle   $\alpha \beta$ will be such that, for
% \textcolor{green}{!!!D!!!March16  "for any i" ?}\textcolor{red}{sim}
each   $i=1,\dots,r$, will exist $z_i \in {\rm Supp}\, \alpha_i \cap {\rm Supp}\, \beta$  such that:
 \begin{eqnarray}\label{consecutives}
 \textrm{$z_i^{\alpha^{l}} \in {\rm Supp}\, \alpha_i \cap {\rm Fix}\, \beta$ and $z_i^{(\alpha\beta)^l}=z_i^{\alpha^l}$, for $l=1,\dots, |{\rm Supp}\, \alpha_i|-2$},
 \end{eqnarray}
then  at least $|{\rm Supp}\, \alpha_i|-1$ consecutive elements in $\alpha_i$ are also consecutive in $\alpha \beta.$\\\\
% {\it {\bf Fact 2:} Notice that $|{\rm Supp}\, \alpha_i \cap {\rm Fix}\, \beta|\geq|{\rm Supp}\, \alpha_i|-2$, for $i=1,\dots,r.$}\\\\
{\it {\bf Fact 2:} Suppose that $\langle \alpha \rangle$ admits a nontrivial system of blocks and  $\beta$ satisfies (\ref{FixB}) and (\ref{consecutives}). If there is a cycle $\alpha_i$   containing consecutive elements
in the same block 
%{\bf !!!D!!!Sep/18 Aqui quer dizer que se tiver pelo menos 2 consecutivos ai o resultado ja vale?} 
% \textcolor{red}{ in which consecutive elements which are in the same block,}
%    \textcolor}{red}{in which {\bf !!!D!!!Sep/18 Aqui quer dizer que se tiver pelo menos 2 consecutivos ai o resultado ja vale?} consecutive elements are in the same block,} 
then this block will be trivial for  $\langle \alpha, \beta \rangle$.}

 \vspace{4pt}
 \noindent This is because by applying $\alpha$ repeatedly on these consecutive elements, the support of $\alpha_i$ will be absorbed by the block. Then, by (\ref{consecutives}), there are consecutive elements in $\alpha \beta$ belonging  to the same block and, by the argument above, the support of $\alpha \beta$ will also be absorbed by the block. Since $\alpha \beta$ is a $d$-cycle  by Fact 1 then  the block will be trivial.
%{\bf !!!D!!!Sep/18 Aqui quer dizer que se tiver pelo menos 2 consecutivos ai o resultado ja vale?} 
% \textcolor{red}{ in which consecutive elements which are in the same block,}
%    \textcolor}{red}{in which {\bf !!!D!!!Sep/18 Aqui quer dizer que se tiver pelo menos 2 consecutivos ai o resultado ja vale?} consecutive elements are in the same block,} 

%\textcolor{green}{!!!D!!!March16 Which block?  Do we need to give  some explanation why Fact 2 holds?}

\vspace{5pt}
\noindent From Fact 1 and  Fact 2,  a proper divisor $u$ of $d$  cannot generate a system with $u$ blocks for our case, satisfying (\ref{FixB}) and  (\ref{consecutives}), if
 there is $c_i>u$ with $\gcd(u,c_i)=1$. To show this, suppose that $\langle \alpha, \beta \rangle$ admits a system with $u$ blocks and that there is $c_i$ such that $u \leq c_i-1$ and $\gcd(u,c_i)=1$. Let $S$ be the subset  of ${\rm Supp}\,\alpha_i$ whose elements are consecutive  in the $d$-cycle $\alpha \beta$, as claimed in  (\ref{consecutives}). Then $c_i-1 \leq |S|\leq c_i$ and we have:
 \vspace{-7pt}
  \begin{itemize}[noitemsep, leftmargin=15pt]
 \item If  $|S|=c_i$  then from Fact 1 and (\ref{consecutives}), there are $x,y \in {\rm Supp}\,\alpha_i$ in the same block such that 
 $x^{{\alpha_i}^u}=x^{{\alpha}^u}=y$. Since $\gcd(u,c_i)=1$, by applying $\alpha$ several times the block will absorb ${\rm Supp}\, \alpha_i$. Then Fact 2 implies that the block has to be trivial.
 \item If $|S|=c_i-1$  and $u<c_i-1$, we conclude that the block has to be trivial analogously to the previous case.
 \item  If $|S|=c_i-1$  and $u=c_i-1\geq 3$. Let $x \in {\rm Supp}\, \alpha_i \setminus S$ then there is $\ell < d/u$ such that 
 $x^{{(\alpha \beta)^{\ell u}} }\in S$ with $x$ and  $x^{{(\alpha \beta)^{\ell u}} }$  in the same block determined by $u$.   If they are consecutive in $\alpha_i$ by Fact 2 we conclude that the block has to be trivial. If they are not consecutive in $\alpha_i$ then $y:=x^{\alpha} =x^{\alpha_i}$ and $z:=(x^{{(\alpha \beta)^{\ell u}} })^{\alpha}=(x^{{(\alpha \beta)^{\ell u}} })^{\alpha_i}$ are in the same block and both  in $S$ with $y^{(\alpha \beta)^v}=z$ and $v<u$. But this is a contradiction with Fact 1.
 \end{itemize}
   Hence:\\
\noindent {\it  {\bf Fact 3:}  a proper divisor $u$ of $d$ could generate a system with $u$ blocks for $\langle \alpha, \beta \rangle$  satisfying (\ref{FixB}) and (\ref{consecutives}),  if for all $c_i$ we have either $c_i \leq u$ or   $\gcd(u,c_i)\neq 1.$}

\vspace{5pt}  
\noindent    We will explore the algebraic factorization of $\mathscr{D}$ (assumption at the beginning of the proof). By Remark \ref{blocos}, the algebraic factorization of $\mathscr{D}$ provides a factorization $d=uw$ and  describes the structure of  a system of  blocks  for  the group $\langle \alpha, \beta \rangle$. % this is possible by  hypothesis  in $iii)$ and Lemma \ref{s<a}.\\
% Since $\mathscr{D}$ is decomposable,
Namely,  for $\mathscr{D}=\{D_1,D_2\}$ we will have 
 $D_1=U_{1}.\mathscr{W}_{1}$ and  $D_2= U_{2}.\mathscr{W}_{2}$, where   $\mathscr{U}=\{U_1,U_2\}$ is a set of partitions of $u$, with $\nu(\mathscr{U})=u-1$,  and  by Proposition \ref{fator}: $\mathscr{W}=\{\mathscr{W}_1,\mathscr{W}_2\}$, is a set of collections of partitions of $w$, with $\nu(\mathscr{W})=w-1$. \\
Suppose  $U_{i}$ with $\kappa_i$ summands and
   $\mathscr{W}_{i}$  with $\kappa_i$ partitions of $w$, then
   $ 2<\kappa_1+\kappa_2=u+1.$ The  inequality
 comes from the hypothesis $\gcd(D_i)=1$ that implies $\kappa_i>1$ and  the equality comes from  $\nu(\mathscr{U})=u-1.$ Moreover, 
  from Corollary \ref{U1} we can suppose
  $\kappa_1<u$ and $ \,U_1=[u_{1_1},\dots,u_{1_{\kappa_1}}] $ with $u_{1_1}> 1$. Hence $\, \mathscr{W}_1=\{W_{1_1},\dots,W_{1_{\kappa_1}}\}$ and 
  \begin{eqnarray}\label{fac}
   \alpha \in U_1.\mathscr{W}_1=\{u_{1_1}.W_{1_1},\dots, u_{1_{\kappa_1}}.W_{1_{\kappa_1}}\}=D_1.
   \end{eqnarray}
{\it  {\bf Fact 4:}  the factorization in (\ref{fac}) defines $u$ blocks of length $w$ for $\langle \alpha \rangle$ (see Remark \ref{blocos}): with the elements in the cycles of $\alpha$ corresponding to the summands in $u_{1_i}.W_{1_i}$,  we define $u_{1_i}$ blocks  of  size $w$  preserved by $\alpha$  and each summand  of $W_{1_i}$ means the number of elements in a block in the   respective cycle, for $i=1,\dots,\kappa_1$. }

\vspace{5pt}
\noindent Moreover:
\vspace{-5pt} 
  \begin{enumerate}[noitemsep, leftmargin=15pt]
\item If $u_{1_{i}}=1$ for some $i \in \{1,\dots,\kappa_1\}$ then the cycles defined by the summands in $u_{1_{i}}.W_{1_i}$ have their supports contained in a block.  Then, by Fact 2, the block is trivial for  $\langle \alpha, \beta \rangle$.
%  \item if  $u_{1_{i}}>1$, for $i=1,\dots, s_1$, and $u_{1_{*}}<x$ 
    \item  If  $u_{1_{i}}>1$, for $i=1,\dots, \kappa_1$ and if $\mathscr{W}_1$ is not trivial, by (\ref{fac}) there exists a cycle  $\alpha_m$ associated to an integer in $u_{1_{j}}.W_{1_j}$ for some $j \in \{1,\dots,\kappa_1\} $, such that $|{\rm Supp}\, \alpha_m| > u_{1_j}$ and  has two elements in the same block by Fact 4. Then, by (\ref{consecutives}) there is  an element $z_{m_1} \in {\rm Supp}\,\alpha_m$ such that $z_{m_1}$ and $z_{m_1}^{(\alpha \beta)^{u_{1_{j}}}}$ are in the same block. By Fact 1 and Fact 3, this block will be trivial  block for $\langle \alpha, \beta\rangle$ unless  $u_{1_{j}}=u$. In this case $\kappa_1=1$ and $U_1=[u]$. But this is a contradiction because $\gcd(D_1)=1.$ Therefore if  $u_{1_{i}}>1$, for $i=1,\dots, \kappa_1$ then $\mathscr{W}_1$ is trivial.
  \end{enumerate}  
     Hence we conclude that:
     
     \vspace{7pt}
 \noindent   {\it  {\bf Fact 5:} the existence of blocks  for $\langle \alpha, \beta \rangle$ satisfying (\ref{FixB}),(\ref{consecutives}) and (\ref{fac}),  is possible only  if   $\mathscr{W}_1$ is  trivial. Then each block, which is known to  have length $w$, contains at most one element of a cycle, and any two cycles which contains elements of the same block they have the same length.} 
 
 \vspace{5pt}
 \noindent  If $\mathscr{W}_1$ is  trivial, then  in (\ref{fac}):
     $W_{1_i}=[1,\dots,1]$ for $i=1,\dots,\kappa_1$, thus:
     \begin{eqnarray}\label{D1}
     D_1=[\underbrace{u_{1_1},\dots,u_{1_1}}_{w}, \underbrace{u_{1_2},\dots,u_{1_2}}_{w},\dots, \underbrace{u_{1_{\kappa_1}},\dots,u_{1_{\kappa_1}}}_{w}]
     \end{eqnarray}
      $r=\kappa_1w$ and, since $\gcd(D_1)=1$, the summands of $D_1$ are all less than $u$. Then, by 
       (\ref{D1D2}), we have:
      \begin{eqnarray}\label{u1's}
       u > u_{1_1} \geq u_{1_2} \geq \dots \geq u_{1_{\kappa_1} }.
       \end{eqnarray}
              
 \noindent    To finish, we provide in this last step the definition of  $\beta$  where  the non-trivial cycles of   $\beta$ must satisfy (\ref{FixB}) and (\ref{consecutives}),
  then show that there is no a system of $u$ blocks for   $\langle \alpha, \beta \rangle$ and finally that  there is no a system of  blocks for   $\langle \alpha, \beta \rangle$.

    \vspace{7pt}
  \noindent     {\bf Fact 6:} Let us define $\beta  =\beta_1 \beta_2 \dots \beta_s \in D_2=[e_{1},e_{2},\dots,e_{s},1,\dots,1]$
  in such a way that their non-trivial cycles  connect  consecutive cycles of $\alpha$, using the order given by (\ref{D1}) and (\ref{u1's}).  Namely:\\
    1)  the  $e_1$-cycle   $\beta_1$ connects  the cycles: $\alpha_1, \dots, \alpha_{e_1}$, which are the first  $e_1$ consecutives cycles of  $\alpha$,\\
        2) the  $e_2$-cycle $\beta_2$ connects the cycles: $\alpha_{e_1},\alpha_{e_1+1},\dots, \alpha_{e_1 + (e_2-1)}$, which are the next $e_2$ consecutive cycles
               of $\alpha$,\\
        3) for  $2<i\leq s$, the  $e_i$-cycle  $\beta_i$ connect the cycles:\\
         $\alpha_{e_1+(e_2-1)+\dots+(e_{i-1}-1)}, \alpha_{e_1+(e_2-1)+\dots+(e_{i-1}-1)+1}, \dots, \alpha_{e_1+(e_2-1)+\dots+(e_{i-1}-1)+ (e_{i}-1)}$, which are $e_i$ consecutives 
         cycles in $\alpha$. 

  \vspace{5pt}
  \noindent     Since $\gcd(D_1)=1$, there exist in  (\ref{u1's}) a strict inequality, $u_{1_j} < u_{1_{j+1}}$ for some $1_j$. Then, by Fact 6,  there exist one and only one  cycle of $\beta$ connecting  the  pair of  cycles of $\alpha$ which correspond to  $u_{1_j},  u_{1_{j+1}}$. Hence, by Fact 5, this cycle of $\beta$  relates two different  blocks for $\langle \alpha \rangle$. We assert that $\beta$ does not  preserve them. By  Fact 5, this cycle of $\beta$  acts on two different blocks, 
call them $B_1$ and  $B_2$, of  $\langle \alpha \rangle$.
If $\beta$ preserve blocks we need to have exactly  
 $ w>1$ cycles of $\beta$ with the same property, to have
$(B_1)^{\beta}=B_2$. But the definition of 
 $\beta$ admits only one cycle of this type. Therefore there is no a system with $u$ blocks for   $\langle \alpha, \beta \rangle$. \\
  In case there is a non-trivial system of blocks for the group $G$ with $u'$ blocks, this implies  that     
there is an   algebraic factorization  of the datum  and $d=u'w'$. So  
$\mathscr{D}=\{D_1,D_2\}=\mathscr{U}'.\mathscr{W}'$, for some  $\mathscr{U}'=\{U_1',U_2'\}$  a set of partitions of $u'$ with $\nu(\mathscr{U}')=u'-1$,   $\mathscr{W}'=\{\mathscr{W}_1',\mathscr{W}_2'\}$   a set of  collections  of partitions of $w'$, with $\nu(\mathscr{W}')=w'-1$, and we have  
  $D_1={U}_{1}'.\mathscr{W}_{1}'$.
 Using the same arguments as in the case of the first decomposition we restrict  to the case where      
 $\mathscr{W}_{1}'$ is  trivial, and we have 
 \begin{eqnarray}\label{1}
      D_1=[c_1, \dots, c_r]=[\underbrace{u_{1_1}',\dots,u_{1_1}'}_{w'}, \underbrace{u_{1_2}',\dots,u_{1_2}'}_{w'},\dots, \underbrace{u_{1_{\ell_1}'},\dots,u_{1_{\ell_1}}'}_{w'}],
     \end{eqnarray}
      where    $d=u'  w'$, $ u'=u_{1_1}'+\dots +u_{1_{\ell_1}}$ and
      \begin{eqnarray}\label{2}
      u'> u_{1_1}' \geq u_{1_2}' \geq \dots \geq u_{1_{\ell_1}' }.
       \end{eqnarray}
       
As in the previous case, there is a strict inequality $u_{1_k}' > u_{1_{k+1}}'$, the same as in   the first decomposition,
  then the  already defined $\beta$ does not preserve  the system of $u'$ blocks of $\langle \alpha \rangle$, hence the result follows.
\end{proof}

  \begin {Rm} The cases covered by Proposition \ref{3cycles-primitive}  are not disjoint of the cases covered by Lemma \ref{C123}. Furthermore, neither statement implies the other.
        \end{Rm}
\begin{Prop}\label{2-indecomp}Let $\mathscr{D}=\{D_{1},D_{2}\}$  be a branch datum such that $\nu(\mathscr{D})=d-1$. If $\gcd(D_{i})=1$ for $i=1,2$ and  $\mathscr{D}\neq \{[2,\dots,2,1],[2,\dots,2,1]\} $ then $\mathscr{D}$ is realizable by an indecomposable branched covering.
\end{Prop}
\begin{proof} Notice that $D_1$ will satisfy the hypothesis either in Proposition \ref{3cycles-primitive}  or  in Lemma \ref{C123}.
In the first case Proposition   \ref{3cycles-primitive} state that   there are permutations $\alpha \in D_1, \beta \in D_2$ such that $\langle \alpha, \beta, \sqrt{\alpha \beta} \rangle$ is  primitive, and in the second case,  from Lemma \ref{C123},  follows  the same conclusion.
 Since the group is transitive,  Proposition \ref{2.1-EKS} implies the existence of a  branched covering of degree $d$, $\phi: \mathbb{R}P^2 \rightarrow \mathbb{R}P^2$,  with 
branch data $\mathscr{D}$ and monodromy group $\langle \alpha, \beta, \sqrt{\alpha \beta}\rangle$. Since the group is primitive,  Proposition \ref{yo} implies that the branched covering $\phi$ is indecomposable.\end{proof}
 
\begin{Cor}\label{AD} Let $\mathscr{D}=\{D_{1},D_{2}\}$  be a branch datum such that $\nu(\mathscr{D})=d-1$.
If $\mathscr{D}$ does not admit an indecomposable realization then either $\gcd(D_1)\neq 1$ or $\mathscr{D}=\{[2,\dots,2,1],[2,\dots,2,1]\}$.
\end{Cor}

\begin{proof}
The branch data which satisfy the hypotheses of the corollary  are  the complement of the data studied in Theorem \ref{2-indecomp}, so the result follows. 
\end{proof}
\subsubsection{Uniquely decomposable}\label{UD}
In this section we show the converse of Corollary \ref{AD}, i.e.  the cases $\{[2,\dots,2,1],[2,\dots,2,1]\}$ and $\{D_1,D_2\}$  with $\gcd(D_1)\neq1$  do not admit realization by indecomposable primitive branched coverings. In terms of \cite{ACDHL} and under our hypothesis,  the case $\gcd(D_1)\neq1$ will imply that: to have a generator with cyclic structure  an ic-partition of type $(\gcd(D_1),d/\gcd(D_1))$, for $1<\gcd(D_1)<d$, guarantees the imprimitivity of the permutation group. This is related with \cite[Theorem 2.5]{ACDHL}, but does not follows from it. 

 \begin{Prop}\label{special} If $d$ is odd and not a prime,  then any realization of the datum  \\
$\mathscr{D}=\{[2,\dots,2,1], [2,\dots,2,1]\}$ is decomposable. 
 \end{Prop}
\begin{proof} Step 1- We claim that  the subgroup generated by $\alpha, \beta$ is isomorphic to either a finite cyclic group or $Z_k \rtimes Z_2$ with the action multiplication by $-1$. To see this since $\alpha^2=\beta^2=1$ follows that the subgroup $\langle \alpha, \beta \rangle$ is a quotient of the infinite dihedral group $Z_2*Z_2=Z\rtimes Z_2$      where the action of $Z_2$ on $Z$ is non trivial. By straighrforward calculation using the short exact sequence:\\ 
\centerline{$1 \to Z \to Z\rtimes Z_2 \to Z_2 \to 1$}\\
 we have that    a normal subgroup of the middle group is also of the form $kZ$  or $kZ\rtimes Z_2$ and the result follows.  
\vspace{4pt}\\
\noindent Step 2- If the subgroup is either  cyclic or $k<d$ then the subgroup is not transitive. In case it is cyclic then we have only one brached point. In the case of $Z_k\rtimes Z_2$ the block of the generators of $Z_k$ are also blocks of the generator of $Z_2$ as   result of the relation $yxy^{-1}=x^{-1}$. 
\vspace{4pt}\\
\noindent Step 3- The subgroup $\langle \alpha, \beta\rangle$ is isomorphic to $Z_d\rtimes Z_2$ from the previous steps. Now any block of the  cyclic part $Z_d$ is also a block for $\beta$ because of the relation  $yxy^{-1}=x^{-1}$    where here $x$ is a generator of $Z_d$ and y a generator of $Z_2$.
 \end{proof}

\begin{Prop}\label{2-trans}
Let $\mathscr{D}=\{D_{1},D_{2}\}$  be a branch datum such that $\nu(\mathscr{D})=d-1$ and    $\gcd(D_1) \neq 1$. Then for every pair of permutations $\alpha \in D_1$ and $\beta \in D_2$,  if $\langle \alpha, \beta \rangle$ is transitive then it is imprimitive. Hence, every realization having  $\langle \alpha, \beta \rangle$  as 
monodromy group is decomposable.
\end{Prop}
\begin{proof}
Let $D_1=[c_1,\dots,c_r]$  and $\alpha=\alpha_1 \dots \alpha_r \in D_1$,  where $\alpha_i$ is  a $c_i$-cycle, for $i=1,\dots r.$ Let $* \in {\rm Supp}\,\alpha_1$ and  denote by $\mathscr{O}_0$ the orbit of $*$ by $\alpha_1^{\gcd(D_1)}$. Since $\gcd(D_1)$ divides $c_1$ then $\mathscr{O}_0$ is a  block  for $\langle \alpha_1 \rangle$, possibly trival: $\mathscr{O}_0=\{*\}$  if $c_1=\gcd(D_1)$.\\
For $D_2=[e_{1},e_{2},\dots,e_{s},1,\dots,1]$ let $\beta=\beta_1\dots \beta_s \in D_2$, with $\beta_j$ a $e_j$-cycle, for $j=1,\dots,s,$ such that $\langle \alpha, \beta \rangle$ is transitive. By Corollary  \ref{b_i}, each  $\beta_j$  has exactly one element in common with  $e_j$ different  cycles of  $\alpha$, 
 for $j=1,\dots,s$ and all cycles of $\alpha$ are attained. 
 
 \vspace{5pt} 
\noindent  Without loss of generality let us suppose that: 
\begin{enumerate}[noitemsep, leftmargin=15pt]
\item    ${\rm Supp}\,\beta_1$ intercepts once the support of each cycle in $Z_1:=\{\alpha_{_1},\dots,\alpha_{_{e_1}}\}.$

\vspace{4pt}
Then
$\alpha \beta_1=\Lambda_1\;\alpha_{e_1+1} \dots \alpha_{r},$
 where $\Lambda_1$  is a $(\sum_{i=1}^{e_1}c_i)$-cycle given by\\
 \centerline{ $(\,\iorarrow[]{\alpha_1} \; \iorarrow[]{\alpha_2}\;\dots\; \iorarrow[]{\alpha_{e_1}}\, )$,}\\
  with $\iorarrow[]{\alpha_{\bullet}}$ denoting the sequence of elements defined by the permutation $\alpha_{\bullet}$ starting at the element in ${\rm Supp}\; \alpha_{\bullet} \cap {\rm Supp}\, \beta_1$.\\
By the hypothesis, $\gcd(D_1)$ divides $\sum_{i=1}^{e_1}c_i$, hence if  $\mathscr{O}_1$ denotes the orbit of $*$ by $\Lambda_1^{\gcd(D_1)}$ then $\mathscr{O}_1$ is a nontrivial block for $\langle \Lambda_1 \rangle$ and for $\langle \prod_{j=1}^{e_1} \alpha_j \rangle$, with
 $\mathscr{O}_0 \subset \mathscr{O}_1$.

 \vspace{5pt}
\item For $j=2,\dots,s$,  $\, {\rm Supp}\,\beta_j$ intercepts once the support of each cycle in  \\
\centerline{ $Z_j:=\Big \{\Lambda_{j-1},\alpha_{_{e_1+(\sum_{\iota=2}^{j-1} e_{\iota}-1)+1}},\dots, \alpha_{_{e_1+(\sum_{\iota=2}^j e_{\iota}-1)}}  \Big \}$.}

 \vspace{5pt}
 Then
$\alpha\beta_1\dots \beta_j= \Lambda_j\;\alpha_{_{e_1+(\sum_{\iota=2}^j e_{\iota}-1)+1}} \dots \alpha_{r},$
where $\Lambda_j$ is a $\Big (\sum_{i=1}^{e_1+(\sum_{\iota=2}^j e_{\iota}-1)}c_i\Big )$-cycle given by \\
\centerline{$\Big (\iorarrow[]{\Lambda_{j-1}} \; \iorarrow[]{\alpha_{_{e_1+(\sum_{\iota=2}^{j-1} e_{\iota}-1)+1}}}\dots\; \iorarrow[]{\alpha_{_{e_1+(\sum_{\iota=2}^j e_{\iota}-1)}}} \Big )$,}\\
  with the sequences $\iorarrow[]{\Lambda_{j-1}}$ and $\iorarrow[]{\alpha_{\bullet}}$  starting at the element in  ${\rm Supp}\; \Lambda_{j-1} \cap {\rm Supp}\; \beta_j$ and ${\rm Supp}\; \alpha_{\bullet} \cap {\rm Supp}\; \beta_j$, respectively. By hypothesis, $\gcd(D_1)$ is a proper divisor of $\sum_{i=1}^{e_1+\sum_{i=2}^{l}(e_i-1)}c_i$, thus if we denote by $\mathscr{O}_j$ the orbit of $*$ by $\Lambda_j^{\gcd(D_1)}$ then $\mathscr{O}_j$ is a nontrivial block for $\langle \Lambda_j \rangle$ and for $\Big \langle \prod_{j=1}^{e_1+(\sum_{\iota=2}^j e_{\iota}-1)} \; \alpha_j \Big  \rangle$, with $\mathscr{O}_{j-1}\subset \mathscr{O}_j$. 
\end{enumerate}
Finally note that $\Lambda_s=\alpha \beta$ and $\mathscr{O}_s$  is a nontrivial block for $\langle \alpha \beta \rangle$ and for $\alpha.$   
 Then  $\langle \alpha, \alpha \beta \rangle=\langle \alpha, \beta \rangle$ is imprimitive. The final part follows from Proposition \ref{yo}.
\end{proof}

From the above result, the case that remains to be studied is:
\begin{equation}\label{LC}
\begin{array}{lll}
\textrm{Let $\mathscr{D}=\{D_{1},D_{2}\}$ be   a collection of partitions of $d \in \mathbb{Z}^{+}$, with:}\\
  \textrm{$\nu(\mathscr{D})=d-1$,
$\nu(D_1)\geq \nu(D_2)>0$ and   $\gcd(D_1)\ne 1$  a proper divisor of $d$.} \\
\textrm{Let $\alpha \in D_1, \beta \in D_2$ be permutations such that
$\langle \alpha, \beta \rangle$ is not transitive but}\\
\textrm{and let  $\omega$ such that $G:=\langle \alpha, \beta, \omega| \alpha \beta= \omega^2\rangle$ is transitive.}
\end{array}
\end{equation}

 \vspace{5pt}
\noindent We will show that  $G$ is imprimitive. Equivalently, 
 the associated   primitive branched covering   is
 decomposable.\\
 \noindent  To show this, we will use as a  strategy the contrapositive of the following proposition:

\begin{Prop}\cite[Proposition 8.7]{Wi}\label{HW} If $G$ is primitive on $\Omega$ and $x$ and $y$ are different points of $\Omega$, then either $G=\langle G_x,G_y \rangle$ or $G$ is a regular  
group of prime degree. \qed
\end{Prop}

\noindent A  permutation group $G$ on $\Omega$ is called {\it regular } if for every $\alpha\in \Omega$, $G_{\alpha}=1$, and it is transitive. \\

\noindent For the case in (\ref{LC}) we have that:
 $d \in \mathbb{Z}^+$ is  not prime (because $\gcd(D_1)$ is  a proper divisor of $d$), $\beta$ is not trivial  (because $\nu(D_2)>0$ ) and  ${\rm Fix}\, \beta \neq \emptyset$ (from Remark \ref{numberrestr} item b). Then $G$ is not a regular group of prime degree. Thus,  in order to show that $G$ is imprimitive, it suffices   to prove that: \\
 \centerline{\it there is $x,\, y \in \Omega$ such that $G\neq \langle G_x,G_y \rangle$.}
 
 \vspace{7pt}
\noindent We will start by derive some properties of $G=\langle \alpha, \beta, \omega| \alpha \beta= \omega^2\rangle$ as result of the fact that it satisfies:
``$\langle \alpha, \beta \rangle$ is not transitive but
 $\langle \alpha, \beta, \omega| \alpha \beta= \omega^2\rangle$ is", as stated in (\ref{LC}).

\begin{Rm}[About $\omega$]\label{omega} 

\vspace{3pt}
\noindent  From the relation  $ \alpha \beta= \omega^2$, once $\omega$ exists,  its cyclic structure 
depends on the cyclic structure of $\alpha \beta.$ Moreover:
  \begin{enumerate}[noitemsep, leftmargin=15pt]
\item Let $H:=\langle \alpha, \beta \rangle$, and let us consider the partition given by the orbits of $H$. Notice that  each of these orbits 
 is defined by  the union of cycles of   $\alpha \beta$. Once $H$ is not transitive, then this partition of the 
cycles of  $\alpha \beta$ contains at least two classes. A cycle of one class need to be connected to a cycle of another class. So this implies that there is 
 a cycle  $\pi_1$ of length $\kappa$,   
and    there exists another cycle $\pi_2$ in $\alpha \beta$, with the same length but contained in another orbit,   such that $\omega$ connects them. This is, if\\
\centerline{$\alpha \beta= \dots \underbrace{(a_1 \dots a_{\kappa})}_{\pi_1}\dots \underbrace{(b_1 \dots b_{\kappa})}_{\pi_2} \dots $}\\
then $\omega$ will intersperse them, i.e.:\\
\centerline{$\omega=\dots ( a_1\; b_1 \dots a_{\kappa} \; b_{\kappa}) \dots  $}
\item Hence, $\omega$ will have two types of cycles:
\begin{itemize} [noitemsep, leftmargin=15pt]
\item type 1 cycles: cycles  connecting different orbits of $\langle \alpha, \beta \rangle$, the set of this type of cycles is not empty,
\item type 2 cycles: cycles whose support is in an orbit  of $\langle \alpha, \beta \rangle$.
\end{itemize}
\item If $\langle \alpha, \beta \rangle$ has $t$ orbits, then $\omega$ has to make $t-1$ connections.
\end{enumerate}
\end{Rm}

\begin{Rm}[About the orbits of $\langle \alpha, \beta \rangle$]\label{orbitas} Let us denote by $\Lambda_0$ the set of orbits of $\langle \alpha \rangle$ and by $\Lambda_j$ the set of orbits of $\langle \alpha, \prod_{i=1}^j\beta_i  \rangle$, recall $\beta=\beta_1\dots \beta_s \in D_2$,  with $\beta_j$ a $e_j$-cycle for $j=1,\dots,s.$ To understand the orbits of $\langle \alpha, \beta\rangle$, notice that:
  \begin{enumerate}[noitemsep, leftmargin=15pt]
  \item $\# \Lambda_0=r.$
  \item The orbits of $\langle \alpha, \beta \rangle$ are
   union of sets which are elements of 
   $\Lambda_0$.
      Hence  the number of orbits of $\langle \alpha, \beta \rangle$ is  $\# \Lambda_s \leq r$. 
    \item Let  $n_j:= \# \{O \in \Lambda_{j-1}: {\rm Supp}\, \beta_j \cap O \neq \emptyset\}$, for $j=1, \dots, s$. Notice that: 
  \begin{enumerate}
  \item $n_j \leq e_j$ and
if $e_j-n_j>0$ then ${\rm Supp}\, \beta_j$ ``multi-intersects" (intersects more than once) the same orbit of $\Lambda_{j-1}$. We will say that $|{\rm Supp}\, \beta_j|-n_j$ is {\it ``the number of multi-intersections"} of ${\rm Supp}\, \beta_j$ on $\Lambda_{j-1}$.
   \item $\# \Lambda_j=\# \Lambda_{j-1}-n_{j}+1$. Thus,
 $ \# \Lambda_s=r+s-\sum_{j=1}^s n_j.$
 \end{enumerate}
 \item From (\ref{e}):
 $r=1-s+\sum_{j=1}^s e_j ,$ then in the equality above we obtain: 
\begin{eqnarray}\label{No-orbitas}
 \# \Lambda_s=1+\sum_{j=1}^s (e_j-n_j).
 \end{eqnarray}
 \noindent  The  non transitivity of $\langle \alpha, \beta \rangle$ implies,  by Corollary \ref{b_i}, on  the existence of a cycle of $\beta$ whose support intersects more than once some cycle of $\alpha$, i.e. $ \# \Lambda_s >1.$
 \end{enumerate}
 \end{Rm}
 \vspace{5pt}

Given a $p$-cycle $\epsilon$ and a permutation $\gamma$ consider the intersection of ${\rm Supp}\, \epsilon$ with each cycle of $\gamma$ (including the trivial ones). We will  say that {\it the  collection of these intersections is a decomposition of $\epsilon$ in $\gamma$}  and the number of sets in the collection is called {\it the division number of ${\rm Supp}\, \epsilon$ in $\gamma$}.
\begin{Lemma}\label{intersection} Let $\epsilon$, $\delta \in {\rm Sym}_{\Omega}$ be a $p$-cycle and   a $q$-cycle, respectively.  If   the cardinality of ${\rm Supp}\, \epsilon \cap {\rm Supp}\, \delta$ is $z>1$  then  in $\epsilon \delta$ the support of $\epsilon$ is divided in at most $z$ cycles.
\end{Lemma}
\begin{proof}
Let $ u \in {\rm Supp}\, \epsilon\, \cap\, {\rm Supp}\, \delta$. Suppose that $\epsilon=(u\; u^{\epsilon}\dots u^{\epsilon^{p-1}})$ and  $\delta=(u\; u^{\delta}\dots u^{\delta^{q-1}})$. Notice that $\delta=(u\; u^{\delta})(u\;u^{\delta^2}) \dots (u\; u^{\delta^{q-1}})$ and that:
\begin{enumerate}[noitemsep, leftmargin=15pt]
\item If $u^{\delta^*} \notin {\rm Supp}\, \epsilon$ then $\epsilon (u\, u^{\delta^*} )=(u\; u^{\epsilon}\dots u^{\epsilon^{p-1}}\, u^{\delta^*})$ 
 is a cycle and its support contains  the  
support of $\epsilon$.
\item If $u^{\delta^{\aleph_1}}= u^{\epsilon^{\ell_{1}}} \in {\rm Supp}\, \epsilon\, \cap\, {\rm Supp}\, \delta$ then  $$\epsilon\,(u\; u^{\delta^{\aleph_1}})=(u \dots (u^{\delta^{\aleph_1}})^{\epsilon^{-1}})(u^{\delta^{\aleph_1}} \dots u^{\epsilon^{p-1}})=(u \dots u^{\epsilon^{\ell_1-1}})(u^{\epsilon^{\ell_1}} \dots u^{\epsilon^{p-1}})$$ and the support of $\epsilon$  intercepts two cycles. Fixed points could occur, it depends on who is $u^{\epsilon^{\ell_1}}$.
\item Let $\aleph_2 > \aleph_1$ such that  $u^{\delta^{\aleph_2}}= u^{\epsilon^{\ell_{2}}} \in {\rm Supp}\, \epsilon\, \cap\, {\rm Supp}\, \delta$ then:
$$\epsilon\,(u\; u^{\delta^{\aleph_1}})(u \; u^{\delta^{\aleph_2}})=(u \dots u^{\epsilon^{\ell_1-1}})(u^{\epsilon^{\ell_1}} \dots u^{\epsilon^{p-1}})(u \; u^{\epsilon
^{\ell_2}})=$$
\begin{equation*}
= 
\begin{cases}
(u \dots u^{\epsilon^{\ell_1-1}}\, u^{\epsilon^{\ell_2}} \dots u^{\epsilon^{\ell_2-1}}),& \textrm{ if $\ell_1 < \ell_2$}\\
(u \dots u^{\epsilon^{\ell_2-1}})(u^{\epsilon^{\ell_2}} \dots u^{\epsilon^{\ell_1-1}})(u^{\epsilon^{\ell_1}} \dots u^{\epsilon^{p-1}}), & \textrm{ if $\ell_2 < \ell_1$.}
\end{cases}
\end{equation*}
\end{enumerate}
Hence if $u^{\delta^{\aleph_i}}= u^{\epsilon^{\ell_{i}}} \in {\rm Supp}\, \epsilon\, \cap\, {\rm Supp}\, \delta$, for $i=1, \dots, z$, with  $\aleph_z > \dots > \aleph_2 > \aleph_1$ and 
$\ell_z< \dots < \ell_2 < \ell_1$   then, considering the cases above, in $\epsilon \delta$ the support of $\epsilon$ intercepts exactly   $z$ cycles. Any other order in $\{\ell_1, \dots, \ell_z\}$ implies the intersection of less than $z$ cycles.
\end{proof}

\begin{Prop}\label{N} Let $\mathscr{D}=\{D_{1},D_{2}\}$  be a branch datum such that $\nu(\mathscr{D})=d-1$  and    $\gcd(D_1) \neq 1$. Let $\alpha \in D_1$ and $\beta \in D_2$. If
 $\langle \alpha, \beta \rangle$ has $t$ orbits and there exists $\omega$ such that $\langle \alpha, \beta, \omega | \omega^2=\alpha \beta \rangle$  is transitive then $\alpha \beta$ is the product of $2t-1$ cycles.
\end{Prop}
\begin{proof} Using the notation in Remark \ref{orbitas}, if $\# \Lambda_s=t$  in (\ref{No-orbitas}) we have that:
\begin{enumerate}[noitemsep, leftmargin=15pt]
\item[(a)]  $\sum_{j=1}^s (e_j-n_j)=t-1$ is the total number of multi-intersections of $\beta$ (see Remark \ref{orbitas} item 3a) and, by Lemma \ref{intersection},  each multi-intersection could produce at most two cycles in $\alpha \beta$. Then $t-1$ multi-intersections  could produce at most $2(t-1)$  cycles in $\alpha \beta$. 
\item[(b)] There is at least  one orbit of $\langle \alpha, \beta \rangle$ without multi-intersections.  Then, this orbit is entirely defined by the support of  a cycle  of $\alpha$, and this cycle is also a cycle of $\alpha \beta$.
\end{enumerate}
From (a) e (b)   the number $N$ of cycles of $\alpha \beta$  satisfy $ N \leq 2(t-1)+1$,  because $\langle \alpha, \beta \rangle$ has $t$ orbits . On the other hand by Remark \ref{omega} (item 3) we have  $ 2(t-1) \leq N.$
Now,  since $\alpha \beta=\omega^2$ is an even permutation and $d$ is odd, then  $N$ is odd and $N=2(t-1)+1$.
  \end{proof}
  
\noindent  Proposition \ref{N} implies that
$\omega$ is the product of $t$ cycles: $t-1$ of them are the type 1 and only one cycle of type 2 (see Remark \ref{omega} item 2). This  exhibit the minimal structure of $\omega$ to guarantee the  transitivity, i.e.: if $\omega=\omega_1\, \dots \, \omega_t$ and $\omega_{i}$ is of type 1, for $i=1,\dots,t-1$, then:
\vspace{-0.5cm}
\begin{eqnarray}\label{minimality}
 \textrm{ $\omega_{i}$ connects exactly two orbits of $\langle \alpha, \beta, \prod_{j=1}^{i-1} \omega_j \rangle$},
 \end{eqnarray} 
 
 \vspace{-0.5cm}
\noindent see Remark \ref{omega}.

   \begin{Prop}\label{dospuntos}
 Let  $G=\langle \alpha, \beta, \omega| \omega^2=\alpha \beta \rangle$ be  a transitive permutation group on $\Omega$, with  $|\Omega|=d$ odd. If $\gcd(\alpha)$ is a proper divisor of $d$ and  the subgroup  $\langle \alpha, \beta \rangle$ acts on $\Omega$ with $t>1$ orbits, then there exist $x,\, y \in \Omega$ such that $G\neq \langle G_x,G_y \rangle$.
 \end{Prop}
 \begin{proof} 
 As in the proof of Proposition \ref{N} item b, we note that there exists at least one orbit of $\langle \alpha, \beta \rangle$ defined by     the support of  a cycle $\pi_1$ of $\alpha \beta$. Let $x \in {\rm Supp}\, \pi_{1}$. From Remark \ref{omega}, there exists another cycle $\pi_2$ in $\alpha \beta$, with the same length of 
$\pi_1$ but contained in another orbit,  such that $\omega$ connects them. Let  $y \in {\rm Supp}\, \pi_2$. \\
 If  $G_x <G_y$  then $\langle G_x, G_y \rangle=\langle G_y \rangle \neq G$, since $\alpha \notin G_y$ because $\gcd(\alpha)\neq 1$. Analogous is  the case $G_y < G_x $.\\
 Let us suppose then that $G_x \not<G_y$ and $G_y \not<G_x$ 
 and let  $J:=\langle G_x, G_y \rangle$. We assert that for all $g \in J$: $x^{g} \neq y$.  To prove this,  we note that:
 \begin{enumerate}[noitemsep, leftmargin=15pt]
 \item $\omega$ connects the orbits $x^{\langle \alpha, \beta \rangle}$ and $y^{\langle \alpha, \beta \rangle}$. If  another $g \in G$ performs the same work but passing through other orbits, by the relation $\omega^2=\alpha \beta$ and $(\ref{minimality})$,  when we write $g$ as a product of the generators, the permutation  $\omega$ has to appear an odd number of times. 
 \item When writing an element of $G_x$ or $G_y$  as a product of generators of $G$, by the relation $\omega^2=\alpha \beta$ and $(\ref{minimality})$,  an even number (possibly zero) of $\omega$ must appear.
 \end{enumerate}
 Then there is no $g\in \langle G_x, G_y \rangle$ satisfying the itens above simultaneously. Hence  $\langle G_x, G_y \rangle$ is not transitive and then $G\neq \langle G_x,G_y \rangle$.
\end{proof} 

\subsubsection{Proof of Theorem \ref{iff2pt} } Let $\mathscr{D}=\{D_{1},D_{2}\}$  be a branch datum such that $\nu(\mathscr{D})=d-1$.\\
For the ``if part",  notice that if $\gcd(D_{1})\neq1$ ($\gcd(D_{12})$ is always $1$) then $d$ is not prime and from Proposition \ref{2-trans}, Proposition  \ref{dospuntos} and Proposition \ref{HW} every realization of $\mathscr{D}$ is given by an imprimitive permutation group. The same result  we obtain if  $\mathscr{D}= \{[2,\dots,2,1],[2,\dots,2,1]\}$, by Proposition \ref{special}.\\
The ``only if part"  follows from Corollary \ref{AD}. \qed

\section{Arbitrary number of branched points}\label{arbpt}
Let $\mathscr{D}=\{D_1,\dots,D_k\}$ be a collection of partitions of $d \in\mathbb{Z}^+$ odd non-prime and $k \geq 3$, such that $\nu(\mathscr{D})=d-1$. Let $\ell_i$ be the number of summands equal to $1$ in $D_i$: 
\begin{eqnarray}\label{Di}
D_i=[e_{i_1}, \dots, e_{i_{r_i}}, \underbrace{1,\dots,1}_{\ell_i}].
\end{eqnarray}

\vspace{-0.5cm}
\noindent Notice that:
\begin{enumerate}[noitemsep, leftmargin=20pt]
 \item[(I)] Since $k \geq 3$ and $\nu(\mathscr{D})=d-1$, there exists at most one $i \in \{1,\dots,k\}$ such that $\nu(D_i)\geq (d-1)/2$. In particular:
 \begin{enumerate} [noitemsep, leftmargin=15pt]
  \item there exists at most one $i $ such that $D_i=[2,\dots,2,1]$,
  \item there exists at most one $i $ such that  $\gcd(D_i)\neq 1$, and in this case $D_i$ will be the partition with biggest defect.
  \end{enumerate}
\item[(II)] If $\nu(D_j)\leq(d-1)/2$ then $\ell_j\neq 0$. Moreover if $\nu(D_j)<(d-1)/2$ then $\ell_j\geq 2$. This is because $D_j$ is a partition of $d$ with at least  $(d+1)/2$ summands.
\end{enumerate}

\noindent   Since  $\nu(D_i)=\sum_{j=1}^{r_i}(e_{i_{j}}-1)$ and $\nu(\mathscr{D})=\sum_{i=1}^k \nu(D_i)=d-1$ then 
\begin{eqnarray}\label{gen-trans}
r_1+ \ell_1=\Big(\sum_{i=2}^k \sum_{j=1}^{r_i }(e_{i_{j}}-1)\Big) +1.
\end{eqnarray}
\noindent Let  $\langle \alpha_1, \dots, \alpha_k, \omega | \omega^2=\prod_{i=1}^k \alpha_i\rangle$  be a monodromy group associated to a realization of $\mathscr{D}$, where $\alpha_i \in D_i$ and $\omega$ are permutations in $\Sigma_d$  (see Proposition \ref{2.1-EKS}).

\begin{Rm}
It follows from the proof of  \cite[Theorem 5.1]{EKS} that the realization of     $\mathscr{D}$ 
 claimed by  Theorem 5.1 is  by a group such that  $\langle \alpha_1, \dots, \alpha_k \rangle$ is transitive, for $\alpha_i \in D_i$. 
  But, as analized in \S\ref{2points} for $k=2$  (see for example  Lemma \ref{3cycles0}  and Proposi\c{c}\~ao \ref{N}), it could happen for $k>2$ that some data also admit a    realization by a group such that  $\langle \alpha_1, \dots, \alpha_k\rangle$ is not transitive.  This shows that we are considering a larger class of realizations  than the one in   \cite{EKS}. 
\end{Rm}

The purpose  of this section is to show the following result:

\begin{theo}\label{final} For $k\geq3$, let   $\mathscr{D}=\{D_1, \dots, D_k\}$ be a collection of partitions of $d \in \mathbb{Z}$   odd  non-prime,  such that $\nu(\mathscr{D})=d-1$. Then $\mathscr{D}$ is realizable by an indecomposable branched covering if, and only if $\gcd(D_i)=1$ for $i=1, \dots, k$.
\end{theo}

\noindent We start by proving the following Lemma:
\begin{Lemma} If $\langle \alpha_1, \dots, \alpha_k \rangle$ is transitive then for any subset $I \subset \{1,\dots,k\}$ we have: $\nu(\prod_{i \in I}\alpha_i)=\sum_{i\in I} \nu(\alpha_i)$. In particular $\nu(\prod_{i =1}^k\alpha_i)=\sum_{i=1}^k \nu(\alpha_i)=d-1$ and  $\prod_{i=1}^k \alpha_i$ is a $d$-cycle. 
\end{Lemma}
\begin{proof} In fact, from (\ref{gen-trans}) we have: 
\begin{eqnarray}\label{gen-trans-alpha}
r_1+ \ell_1= \sum_{i=2}^k (|{\rm Supp}\, \alpha_i| - r_i) +1,
\end{eqnarray}
which means that:
for $i=2, \dots, k$,  the support of $\alpha_i$ is used to make $|{\rm Supp}\, \alpha_i| - r_i$ connections in the set of  cycles of  $\prod_{j=1}^{i-1} \alpha_j.$ Hence $ \prod_{j=1}^{i} \alpha_j $ has 
$r_1+ \ell_1-\Big(\sum_{j=2}^{i}  (|{\rm Supp}\, \alpha_j| - r_j)\Big)$ cycles and, by (\ref{gen-trans-alpha}), $ \prod_{j=1}^{k} \alpha_j $ will have only one cycle.
\end{proof}

So follows from the lemma above  that: $\nu(\prod_{i=1}^k \alpha_i)<\sum_{i=1}^k \nu(\alpha_i)=d-1$,    if   $\langle \alpha_1, \dots, \alpha_k\rangle$ is not transitive. 

\vspace{5pt}

\noindent Let us suppose that:
\begin{eqnarray}\label{ordem}
 \nu(D_1) \geq \nu(D_2)\geq \dots \geq \nu(D_k).
 \end{eqnarray}
 
 \vspace{8pt}

 {\it Proof of   Theorem \ref{final}.}
 
{\bf If $\gcd(D_i)=1$, for $i=1,\dots,k$} \\ Let us consider a monodromy group such that $\langle \alpha_1, \dots, \alpha_k \rangle$ is transitive (see \cite{EKS}, proof of Theorem 5.1).  Let $\Lambda:=\prod_{i=2}^k \alpha_i$ and let $L$ be the partition of $d$ determined by the cyclic structure of $\Lambda$.  From  $\nu(\mathscr{D})=d-1$ and (III) we have $\nu(L)=\sum_{i=2}^k \nu(\alpha_i)$. 
\begin{enumerate}[noitemsep, leftmargin=15pt]
\item If $\nu(D_1)\geq (d-1)/2$ and $D_1 \neq [2,\dots,2,1]$ (in this case the inequality is strict):  then $\nu(L) \leq (d-3)/2$. Then $L \neq [2,\dots,2,1]$ and    $\gcd(L)=1$,  see (II).

\item If  $D_1=[2,\dots,2,1]$: by (I)(a) necessarily  $D_j\neq [2,\dots,2,1]$ for $j=2,\dots,k$.  In this case $\nu(L)=\sum_{i=2}^k \nu(\alpha_i)= (d-1)/2$ and we assert that there is a choice of the $\alpha_i$'s such that $L$ has a summand bigger than $2$.  In fact, since $D_j$ has at least one summand equal to $2$, for $j=2, \dots, k$, it is enough in the  construction in (III)   to define $\alpha_j$, for $j=3, \dots, k$, such that its support  connects the cycles of $\prod_{l=2}^{j-1} \alpha_l$  with higher cardinality support. Hence,  $ \prod_{j=2}^k \alpha_j $ has a cycle bigger than 2 and $L\neq [2,\dots,2,1].$

\item If $\nu(D_1)<(d-1)/2$: then, from (\ref{ordem}),  $D_i\neq [2,\dots,2,1]$ for $i=1,\dots,k$. In this case $\nu(L)\geq (d+1)/2$ hence $L\neq [2,\dots,2,1]$. We assert that there is a choice of the $\alpha_i$'s such that $L$ has a summand equals to $1$. In fact, let $\alpha_2 \in D_2=[e_{2_1}, \dots, e_{2_{r_2}}, \underbrace{1,\dots,1}_{\ell_2}].$ Let us suppose that $e_{_{2_1}}\geq \dots \geq e_{_{2_{r_2-1}}}\geq e_{_{2_{r_2}}}>1$. By (II) we have $\ell_2 \geq 2$ then by doing the  construction in (III) for $\alpha_2$ such that for $j=3, \dots, k$ the permutation $\alpha_j$ connects the cycles of $\prod_{l=2}^{j-1} \alpha_l$ with bigger support, we obtain from $\nu(\mathscr{D})=d-1$, that $ \prod_{j=2}^k \alpha_j $ is the product of $(r_2-\ell_2)-\Big(\sum_{i=3}^{k} \Big(\sum_{l=1}^{r_i }(e_{i_{l}}-1)\Big)\Big)=\sum_{l=1}^{r_1}(e_{1_l}-1)+1>2$ cycles. Then, by the choices, at least one $1$-cycle of $\alpha_2$ is a factor. Hence $\gcd(L)=1.$
\end{enumerate}
Whatever the case,  $\overline{\mathscr{D}}:=\{D_1,L\} \neq \{[2,\dots,2,1], [2,\dots,2,1]\}$ is a collection of partitions of $d$ such that  $\nu(\overline{\mathscr{D}})=d-1$ with $\gcd(D_1)=\gcd(L)=1.$ Then, by Proposition \ref{2-indecomp},  $\mathscr{\overline{D}}$ is realizable by an indecomposable branched covering. This is, there exist permutations $\delta \in D_1, \lambda \in L$ and $\omega \in \Sigma_d$ such that $H:=\langle \delta, \lambda, \omega | \omega^2=\delta \lambda \rangle$ is primitive. Moreover, since $\Lambda$ and $\lambda$ are conjugated, there exists $\eta \in \Sigma_d$ such that $\lambda= \eta^{-1} \Lambda \eta$.\\
 Let $G:=\langle \delta, \eta^{-1} \alpha_2 \eta, \dots, \eta^{-1} \alpha_k \eta , \omega  \rangle,$
notice that:
\begin{itemize}[noitemsep, leftmargin=20pt]
\item[(i)] $\delta \in D_1$ and $\eta^{-1} \alpha_i \eta \in D_i$, for $i=2, \dots, k$.
\item[(ii)] $\delta \big( \prod_{i=2}^k \eta^{-1} \alpha_i \eta \big)=\delta (\eta^{-1} \Lambda \eta)=\delta \lambda = \omega^2 $.
\item[(iii)] $H < G$ and the primitivity of $H$ implies the primitivity of $G$.
\end{itemize}

{\bf If  there exists $i$ such that $\gcd(D_i)\neq1$}\\
From (I)(b) and (\ref{ordem}) we have $i=1$,   $\gcd(D_1)\neq 1$ and $\ell_1=0.$
\begin{enumerate}[noitemsep, leftmargin=15pt]
\item  If  $\langle \alpha_1, \dots, \alpha_k\rangle$ is transitive, then it is imprimitive. In fact: from $\nu(\mathscr{D})=d-1$ we have that ${\rm Supp}\, \alpha_i \cap {\rm Supp}\, \alpha_j = \emptyset$ for $i\neq j \geq 2$  and we can adapt Proposition \ref{2-trans} to $\overline{\mathscr{D}}=\{D_1, L\}$, where $L$ is defined as in the case above.  Hence $\langle \alpha_1, \prod_{i=2}^k \alpha_i \rangle$ will be imprimitive. Moreover, from the construction  in (III),  we have that ${\rm Supp}\, \alpha_j$ intersects $|{\rm Supp}\, \alpha_j|$ different cycles of $\prod_{i=1}^{j-1}\alpha_i$ 
 then  $\gcd(D_1)$ divides $\gcd(\prod_{i=1}^{j-1}\alpha_i)$ and each non-trivial cycle of $\alpha_j$ has support contained in a block of $\langle \alpha_1, \prod_{i=2}^k \alpha_i \rangle$. Hence $\langle \alpha_1, \prod_{i=2}^k \alpha_i, \alpha_2, \dots, \alpha_k \rangle=\langle \alpha_1,  \alpha_2, \dots, \alpha_k \rangle$ is imprimitive.  

\item  If  $\langle \alpha_1, \dots, \alpha_k\rangle$ is not  transitive, with $t>1$ orbits, but there exists $\omega$ such that the permutation group $\langle \alpha_1, \dots , \alpha_k, \omega |\alpha_1 \dots \alpha_k=\omega^2 \rangle$ is transitive, we assert that the group is imprimitive. The proof is essentially the same that the one for  $k=2$ and we write the details in the Appendix. \qed
\end{enumerate}

\section*{Appendix}\label{ap}
This appendix will be devoted to prove  a generalization  of Proposition \ref{dospuntos}. For this will suffice to show that Remark \ref{omega} and Remark \ref{orbitas} are satisfied for a general case,  as explained below.
As we shall see, the arguments are essentially the same. 

Let $\mathscr{D}=\{D_1,\dots,D_k\}$ be a collection of partitions of $d \in\mathbb{Z}^+$ odd non-prime and $k \geq 3$, such that $\nu(\mathscr{D})=d-1$ and $\gcd(D_1)\neq 1$. Let $\ell_i$ be the number of summands equal to $1$ in $D_i$: 
\begin{eqnarray}\label{general}
D_i=[e_{i_1}, \dots, e_{i_{r_i}}, \underbrace{1,\dots,1}_{\ell_i}] \textrm{ and } \alpha_i=\alpha_{i_1}\dots\alpha_{i_{r_i}} \in D_i, \textrm{ with } \alpha_{i_j} \textrm{ a } e_{i_j}\textrm{-cycle}
\end{eqnarray}
such that $H=\langle \alpha_1, \dots, \alpha_k\rangle$ is not  transitive. Suppose that  $t>1$ is the number of orbits $H$,
and  that there exists $\omega$ such that the permutation group $\langle \alpha_1, \dots , \alpha_k, \omega |\alpha_1 \dots \alpha_k=\omega^2 \rangle$ is transitive.\\

\noindent {\bf Remark \ref{omega}G} (About $\omega$).  This  remark is the same as   Remark \ref{omega} where we write  $H:=\langle \alpha_1,\cdots  \alpha_k \rangle$ and $\alpha\beta=\alpha_1\cdots  \alpha_k.$  The verification  is completely similar to the  one for   Remark \ref{omega}.

\vspace{5pt}
\noindent {\bf Remark \ref{orbitas}G}
 (About the orbits of $\langle \alpha_1,\cdots, \alpha_k \rangle$).  With the notation in (\ref{general}), for $i=1,\dots, k$ and $j=1,\dots,r_i$:
\begin{enumerate}[noitemsep, leftmargin=15pt] 
\item  Let $\Lambda_1$ be the set of orbits of $\langle \alpha_1 \rangle$;
\item  For $i \geq 2$, let $\Lambda_{i_j}$ be the set of orbits of $\langle \alpha_1,\dots, \alpha_{i-1}, \prod_{l=1}^j \alpha_{i_l} \rangle $. Notice that:
\begin{enumerate}[noitemsep, leftmargin=15pt]
\item $\Lambda_i:=\Lambda_{i_{r_i}}$ is the set of orbits of $\langle \alpha_1, \dots, \alpha_i \rangle$, 
\item With the convention  $\Lambda_{i_0}:= \Lambda_{i-1}$, set $n_{i_j}:=\#\{O \in \Lambda_{i_{j-1}} : {\rm Supp}\, \alpha_{i_j} \cap O \neq \emptyset\}$, then $\# \Lambda_{i_j}=\# \Lambda_{i_{j-1}}-n_{i_j}+1,$

   \item  If $n_i:= \sum_{j=1}^{r_i} n_{i_j}$ then  $\#\Lambda_i = \#\Lambda_{i-1}-n_i+r_i$ and:
   
   \vspace{-0.65cm}
      \begin{eqnarray}\label{t}
 t:= \# \Lambda_{k}=r_1-\sum_{j=2}^k n_j+ \sum_{j=2}^k r_j,
 \end{eqnarray}
\item If $|{\rm Supp}\, \alpha_i|-n_i>0$ then ${\rm Supp}\, \alpha_i$ intersects more than once  same orbit of $\Lambda_{i-1}$. We will say that $|{\rm Supp}\, \alpha_i|-n_i$ is {\it ``the number of multi-intersections"} of ${\rm Supp}\, \alpha_i$ on $\Lambda_{i-1}$;
 \end{enumerate} 
\item  By replacing (\ref{gen-trans}) in (\ref{t}), with $\ell_1=0$ because $\gcd(D_1)\neq 1$,  we obtain:

  \vspace{-0.65cm}
 \begin{eqnarray}\label{t-orbitas}
 t= \# \Lambda_{k}=\sum_{i=2}^k \Big(|{\rm Supp}\, \alpha_i|-n_i \Big) +1.
 \end{eqnarray}
\end{enumerate}

\begin{Prop}\label{N-g}[Generalization of Proposition \ref{N}] Let $\mathscr{D}=\{D_{1},\dots,D_{k}\}$  be a branch datum such that $\nu(\mathscr{D})=d-1$  and    $\gcd(D_1) \neq 1$. Let $\alpha_i \in D_i$, for $i=1, \dots,k.$  If $\langle \alpha_1,\dots, \alpha_k \rangle$ has $t>1$ 
orbits and  
 there exists $\omega$ such that $\langle \alpha_1,\dots, \alpha_k, \omega | \omega^2=\alpha_1 \dots \alpha_k \rangle$  is transitive then
  $\alpha_1\dots \alpha_k$ is the product of $2t-1$ cycles.
\end{Prop}
\begin{proof} The proof follows by replacing in the proof of Proposition \ref{N}: (\ref{No-orbitas}) by (\ref{t-orbitas}), Remarks \ref{omega} and  \ref{orbitas} by Remarks \ref{omega}G and  \ref{orbitas}G, respectively.
  \end{proof}
\noindent  Proposition \ref{N-g} implies that
$\omega$ is the product of $t$ cycles: $t-1$ of them are the 
type 1 and only one cycle of type 2 (see Remark \ref{omega}G). 
This  exhibit the minimal structure of $\omega$ to  guarantee the transitivity, i.e.: if $\omega=\omega_1\, \dots \, \omega_t$ and $\omega_{i}$ is of type 1, for $i=1,\dots,t-1$, then:
\vspace{-0.65cm}
\begin{eqnarray}\label{MinimalityG}
 \textrm{ $\omega_{i}$ connects exactly two orbits of $\langle \alpha_1, \prod_{i=2}^k \alpha_i, \prod_{j=1}^{i-1} \omega_j \rangle$},
 \end{eqnarray} 
 see Remark \ref{omega}G.

   \begin{Prop}\label{dospuntosG}[Generalization of Proposition \ref{dospuntos}] 
 Let  $G=\langle \alpha_1,\dots, \alpha_k, \omega| \omega^2=\alpha_1 \dots  \alpha_k \rangle$ be  a transitive permutation group on $\Omega$, with  $|\Omega|=d$. If $\gcd(\alpha_1)$ is a proper divisor of $d$ and  the subgroup  $\langle \alpha_1, \dots,  \alpha_k \rangle$ is not transitive, then there exist $x,\, y \in \Omega$ such that $G\neq \langle G_x,G_y \rangle$.
 \end{Prop}
 \begin{proof}
 The proof follows by replacing in the proof of Proposition \ref{dospuntos}: Proposition \ref{N} by Proposition \ref{N-g}, Remark \ref{omega} by Remark \ref{omega}G and (\ref{minimality}) by (\ref{MinimalityG}). 
\end{proof} 

\begin{Cor} Under the hypothesis of the Proposition \ref{dospuntosG} we have the the group $G$ is  imprimitive.
\end{Cor}
\begin{proof} Follows form the proposition above and Proposition \ref{HW}.
\end{proof}

Departamento de Matem\'atica DM-UFSCar

Universidade Federal  de S\~ao Carlos

Rod. Washington Luis, Km. 235. C.P 676 - 13565-905 S\~ao Carlos, SP - Brasil

nvbedoya@ufscar.br

\vspace{0.5cm}

Departamento de Matem\'atica

Instituto de Matem\'atica e Estat\'istica

Universidade de S\~ao Paulo

Rua do Mat\~ao 1010, CEP 05508-090, S\~ao Paulo, SP, Brasil.

dlgoncal@ime.usp.br


\begin{thebibliography}{99}



\bibitem{ACDHL} J. Ara\'ujo, J. P. Ara\'ujo, P. J. Cameron, T. Dobson, A. Hulpke, P. Lopes,  \emph{ Imprimitive permutations in primitive groups}, {\it Journal of Algebra} \textbf{486} 15,
 396-416 (2017)

\bibitem {BN0} N. A.  V. Bedoya,  \emph{Revestimentos ramificados e o problema da decomponibilidade},  Tese de doutorado IME-USP-Junho-2008.


\bibitem {BN&GDL1} N. A.  V. Bedoya, D. L. Gon\c{c}alves,  \emph{Decomposability problem on  branched coverings}, 
 {\it  Matematicheskii Sbornik-Sbornik Mathematics} \textbf{201}  12, 3-20 (2010).







\bibitem {BN&GDL2} N. A.  V. Bedoya and D. L. Gon\c{c}alves,  \emph{Primitivity of monodromy groups of branched coverings: a non-orientable case},  
{\it  JP J. Geom. Topol.} \textbf{12}  no. 2, 219-234 (2012).


\bibitem {BN&GDL3} N. A.  V. Bedoya, D. L. Gon\c{c}alves  and  E. Kudryavtseva,   Indecomposable branched coverings over the projective plane by surfaces M with $\chi(M)\leq 0$,
{\it  J. Knot Theory Ramifications}, \textbf{27}  no. 5, 1850030, 23 pp. (2018).


\bibitem {BGZ} S. Bogaty{\v{i}}; D. L. Gon\c calves,  H. Zieschang, \emph{The
  minimal number of roots of surface mappings and quadratic equations in free
  groups}. Mathematische Zeitschrift \textbf{236}, 419-452 (2001).              



\bibitem {DM} J. D. Dixon,  B. Mortimer,  \emph{Permutation groups}, {\it Graduate
Texts in Mathematics}, \textbf{163}, (1996)

\bibitem {EKS} A. L. Edmonds, R. S. Kulkarni, R. E. Stong,  \emph{Realizability of
  branched coverings of surfaces},  {\it Trans. Amer. Math. Soc.}, \textbf{282},
 (1984)

\bibitem {Ez} C. L. Ezell,  \emph{Branch point structure of covering maps onto
nonorientable surfaces}, {\it Trans. Amer. Math. Soc.}, \textbf{243}, 123-133 (1978)



\bibitem {Hu} D. H. Husemoller,  \emph{Ramified coverings of {R}iemann surfaces}, {\it
Duke Math. J.}, \textbf{29}, 167-174 (1962)

\bibitem {Lo} Pedro Lopes,   \emph{Permutations which make transitive groups primitive},
 {\it  Cent. Eur. J. Math.}, DOI: 10.2478/s11533-009-0050-3  \textbf{7} (4), 650-659 (2009). 









\bibitem {Wi}  H. Wielandt,  \emph{Finite Permutation Groups},  {\it  Academic Press Inc.}  (1964) 

\end{thebibliography}
\end{document}